\documentclass[12pt]{amsart}

\usepackage{amsmath}
\usepackage{caption}
\usepackage[curve]{xypic}
\usepackage{cite}
\usepackage{enumitem}
\usepackage{color}
\usepackage{url}
\usepackage[usenames,dvipsnames]{xcolor}
\usepackage{graphicx}
\usepackage{tikz-cd}
\usepackage[all]{xy}

\makeatletter 
\def\@cite#1#2{{\m@th\upshape\bfseries%
[{#1\if@tempswa{\m@th\upshape\mdseries, #2}\fi}]}}
\makeatother 

\theoremstyle{plain}
\newtheorem{thm}{Theorem}[section]
\newtheorem{cor}[thm]{Corollary}
\newtheorem{prop}[thm]{Proposition}
\newtheorem{lem}[thm]{Lemma}

\theoremstyle{definition}
\newtheorem{defn}[thm]{Definition}
\newtheorem{war}[thm]{Warning}
\newtheorem{ex}[thm]{Example}

\newtheorem{conj}[thm]{Conjecture}
\theoremstyle{remark}
\newtheorem{rem}[thm]{Remark}

\numberwithin{equation}{subsection}
\renewcommand{\bold}[1]{\medskip \noindent {\bf #1 }\nopagebreak}

\newcommand{\nc}{\newcommand}
\newcommand{\rnc}{\renewcommand}
\newcommand{\e}{\varepsilon}

\newcommand{\ann}[1]{}
\newcommand{\eb}[1]{\emph{\textbf{#1}}}

\nc\bA{\mathbb{A}}
\nc\bB{\mathbb{B}}
\nc\bC{\mathbb{C}}
\nc\bD{\mathbb{D}}
\nc\bE{\mathbb{E}}
\nc\bF{\mathbb{F}}
\nc\bG{\mathbb{G}}
\nc\bH{\mathbb{H}}
\nc\bI{\mathbb{I}}
\nc{\bJ}{\mathbb{J}} 
\nc\bK{\mathbb{K}}
\nc\bL{\mathbb{L}}
\nc\bM{\mathbb{M}}
\nc\bN{\mathbb{N}}
\nc\bO{\mathbb{O}}
\nc\bP{\mathbb{P}}
\nc\bQ{\mathbb{Q}}
\nc\bR{\mathbb{R}}
\nc\bS{\mathbb{S}}
\nc\bT{\mathbb{T}}
\nc\bU{\mathbb{U}}
\nc\bV{\mathbb{V}}
\nc\bW{\mathbb{W}}
\nc\bY{\mathbb{Y}}
\nc\bX{\mathbb{X}}
\nc\bZ{\mathbb{Z}}
\nc\cA{\mathcal{A}}
\nc\cB{\mathcal{B}}
\nc\cC{\mathcal{C}}
\rnc\cD{\mathcal{D}}
\nc\cE{\mathcal{E}}
\nc\cF{\mathcal{F}}
\nc\cG{\mathcal{G}}
\rnc\cH{\mathcal{H}}
\nc\cI{\mathcal{I}}
\nc{\cJ}{\mathcal{J}} 
\nc\cK{\mathcal{K}}
\rnc\cL{\mathcal{L}}
\nc\cM{\mathcal{M}}
\nc\cN{\mathcal{N}}
\nc\cO{\mathcal{O}}
\nc\cP{\mathcal{P}}
\nc\cQ{\mathcal{Q}}
\rnc\cR{\mathcal{R}}
\nc\cS{\mathcal{S}}
\nc\cT{\mathcal{T}}
\nc\cU{\mathcal{U}}
\nc\cV{\mathcal{V}}
\nc\cW{\mathcal{W}}
\nc\cY{\mathcal{Y}}
\nc\cX{\mathcal{X}}
\nc\cZ{\mathcal{Z}}

\newcommand{\bk}{{\mathbf{k}}}

\nc{\dmo}{\DeclareMathOperator}
\rnc{\Re}{\operatorname{Re}}
\rnc{\Im}{\operatorname{Im}}
\rnc{\span}{\operatorname{span}}
\dmo{\rank}{rank}
\dmo{\End}{End}
\dmo{\Hom}{Hom}
\dmo{\Jac}{Jac}
\dmo{\Id}{Id}
\dmo{\Ann}{Ann}
\dmo{\Area}{Area}
\dmo{\CP}{\bC P^1}

\title{The boundary of an affine invariant submanifold}
\author[Mirzakhani]{Maryam~Mirzakhani}
\author[Wright]{Alex~Wright}
\begin{document}
\maketitle


\vspace{1cm}



\section{Introduction}\label{S:intro}

\ann{A: All referee comments have been addressed. Some additional small changes have been made, and some of subsection 9.2 has been rewritten. The referee comments and the authors' description of changes are in the margins.} 

\bold{Motivation and main result.} Let $\mathcal{H} (\kappa)$ be a stratum of the moduli space of Abelian differentials,
i.e. the space of pairs $(X,\omega)$ where $X$ is a Riemann surface
of genus $g$ and $\omega$ is a holomorphic $1$-form on $X$  whose zeros have
multiplicities $\kappa_1, \ldots, \kappa_s$ and $\sum_{i=1}^s \kappa_i=2g-2$. The form $\omega$ defines a
 flat metric on $X$ with conical singularities at the zeros
of $\omega$. Such $(X,\omega)$ are called {\em translation surfaces}, and can be also described using polygons up to cut and paste.

The $GL(2, \bR)$  orbit  of any translation surface equidistributes in some affine invariant submanifold, which governs most dynamical and geometric problems on the surface. Besides the obvious ones such as the entire moduli space, affine invariant submanifolds parameterize families of translation surfaces with exceptional flat geometry and dynamics, give rise to amazing families of algebraic curves, and witness unlikely intersections in the moduli space of principally polarized Abelian varieties. Each affine invariant submanifold is a rare gem, and their classification is now a central problem in dynamics on moduli spaces with rich connections to many other areas of mathematics. 

For the connection to flat geometry and dynamics see, for example, \cite{EMa, V, LeWe, Wcyl, AD,EsCh}, and for the connection to Abelian varieties see, for example, \cite{Mc, Mc2, M, M2, Fi1, Fi2, KM, MZag}. 

There are only countably many affine invariant submanifolds \cite{EMM, Wfield}, and we hope for strong classification results based on inductive arguments using degenerations of translation surface. 

Such arguments should take place in a  partial compactification $\overline{\cH}$ of the stratum (moduli space of translation surfaces) compatible with degenerations obtained by expressing translation surfaces in terms of polygons and letting edge lengths go to zero.  In\ann{A: Changed wording slightly (removed ``for simplicity").} the introduction we will restrict to the subspace $\overline{\cH}_{conn}\subset \overline{\cH}$ of connected (single component) translation surfaces. 

Let $\partial \cM$ denote the boundary of $\cM$ in $\overline{\cH}$. By Eskin-Mirzakhani-Mohammadi \cite{EM, EMM}, $\partial \cM \cap \overline{\cH}_{conn}$  is a finite union of affine invariant submanifolds $\cM_i$. The main result in this paper is to compute their tangent spaces. 

Suppose $\cM_i$ is contained in the stratum $\cH_i\subset \overline{\cH}$. At points in $\cH$ near $\cH_i$, we show that the tangent space of the boundary stratum $\cH_i$ is naturally identified with a subspace of the tangent space to $\cH$. 

\begin{thm}[Main Theorem, special case]\label{T:main}
Let $\cM_i$ be one of the finitely many affine invariant submanifolds whose union is $\partial \cM \cap \overline{\cH}_{conn}$.
The tangent space of $\cM_i$ is given by the intersection of the tangent space to $\cM$ with the tangent space to the boundary stratum. 
\end{thm}

 This formula is valid for a sequence of points of $\cM$ converging to the boundary stratum $\cH_i$.  See Section \ref{S:multi} for more precise statements, which allow multicomponent degenerations. The multicomponent case is more complicated, in part because closed $GL(2, \bR)$ invariant sets of multicomponent surfaces are not always finite unions of affine invariant submanifolds, and some of our results in the multicomponent case are conditional on a  version of \cite{EM, EMM} for multicomponent surfaces (see Conjecture \ref{C:MEMM}).\ann{A: Added clarification.} 

Two important invariants of an affine invariant submanifold are rank and  field of definition \cite{Wcyl, Wfield}. The rank is a measure of size related to the flexibility of disjoint sets of cylinders on surfaces in $\cM$, and the field of definition is related to the arithmetic complexity of the linear equations locally defining $\cM$. (This is not the same thing as the field of definition of an algebraic variety, which is related to the polynomial equations globally defining the variety.) As an example we mention that closed $GL(2, \bR)$ orbits have rank 1, and have field of definition $\bQ$ if and only if they contain a square-tiled surface. 

\begin{cor}[Consequences of Main Theorem, special case]\label{C:main}
Each $\cM_i$ has smaller dimension than $\cM$ and rank at most that of $\cM$. The field of definition of $\cM_i$ is  equal to  that of $\cM$. 
\end{cor} 

Our partial compactification $\overline{\cH}$ will be described in more detail in Section \ref{S:multi}. It can be obtained as  the bundle of finite area Abelian differentials over the Deligne-Mumford compactification, modulo zero area subsurfaces, with zeros and possibly other points marked.  For genus $g$ surfaces, this compactification will include surfaces of all genera at most $g$, and hence the boundary of an affine invariant submanifold will typically consist of many pieces of different genera. 

 \bold{Applications.} If $\cM$ is an affine invariant submanifold of genus $g$ translation surfaces, then the rank of $\cM$ is at most $g$. If its rank is equal to $g$, it is called full rank. 

\begin{thm}\label{T:HW}
Any affine invariant submanifold of full rank is equal to a connected component of a stratum, or the hyperelliptic locus in a connected component of a stratum. 
\end{thm}\ann{R: It would nice to give a sketch of proof of Theorem 1.3 since it is a very strong application of the method developed in the present paper.\\A: Added citation.}

The proof of Theorem \ref{T:HW} appears in \cite{FullRank} and uses Theorem \ref{T:main} and other results of this paper. Theorem \ref{T:HW} can be used to simplify the proof of a theorem of \cite{HW}, see \cite[Remark 5.7]{HW}.
   
Paul Apisa has used Theorem \ref{T:main} and and other results of this paper in his recent proof that all orbit closures of rank greater than 1 in  hyperelliptic connected components of strata arise from covering constructions \cite{Apisa}.

\bold{Idea of proof and additional results.} The collection of results we develop to prove Theorem \ref{T:main} may be as interesting as Theorem \ref{T:main} itself. 

Our strategy starts with the Cylinder Deformation Theorem \cite{Wcyl}, which appears as Theorem \ref{T:CDT} below. Every translation surface has infinitely many cylinders \cite{MasurClosed}. Shearing a collection of parallel cylinders while leaving the rest of the surface unchanged gives a path of  translation surfaces called a cylinder twist. The Cylinder Deformation Theorem states that cylinder twists given by equally shearing  the full set of cylinders in some direction on a translation surface $(X,\omega)$ remain in any affine invariant submanifold containing $(X,\omega)$.  

In an effort to understand affine invariant submanifolds geometrically, one could consider, for each $(X, \omega)$, the span $E(X,\omega)$ of derivatives of such cylinder twists. This $E(X,\omega)$ is a subspace of the tangent space to the stratum at $(X,\omega)$. By the Cylinder Deformation Theorem, the dimension of $E(X,\omega)$ is a lower bound for the dimension of any affine invariant submanifold containing $(X,\omega)$.   \ann{R: It is not very clear what $E(X,\omega)$ is.\\A: Clarified that we are considering derivatives of twists. Clarified that $E(X,\omega)$ is a subspace of the tangent space to the stratum. Added reference to Remark \ref{R:E}.} However this $E(X,\omega)$ is  very behaved: In Remark \ref{R:E} we observe that it is nowhere continuous. 

Nonetheless, we will show that the tangent space to the boundary $\cM_i$ is generated by  cylinder twists that can still be performed on nearby surfaces in the interior $\cM$. This implies that, once the tangent space to the boundary stratum $\cH_i$ is  identified with a subspace of the tangent space to the stratum $\cH$, the tangent space to the boundary affine invariant submanifold $\cM_i$ is contained in the  tangent space to $\cM$. 

To do this, we will have to define a larger class of twists than those generating $E(X,\omega)$, by using  knowledge of $\cM$. Our larger class of twists need not remain in all affine invariant submanifolds containing $(X,\omega)$, only $\cM$. 

Recall that two cylinders on a surface $(X,\omega)\in\cM$ with core curves $\alpha, \beta$ are said to be $\cM$-parallel if there is a linear equation ${\int_{\alpha}\omega = c\int_{\beta}\omega}$  for some $c\in \bR$ that holds locally on $\cM$. The Cylinder Deformation Theorem states  that cylinder twists given by equally shearing an equivalence class of $\cM$-parallel cylinders remain in $\cM$. 

Suppose that a surface $(X,\omega)\in \cM$ has exactly two cylinders in some direction, but their ratio is not in the set of ratios of $\cM$-parallel cylinders. The Cylinder Deformation Theorem then guarantees that the twist in each one of these individual cylinders  remains in $\cM$. Such twists will be the prototypical examples of a class of twists that we call recognizable, a notion that depends $\cM$. The key properties of recognizable twists, defined in Section \ref{S:Rec}, are  

\begin{enumerate}
\item at every point of $\cM$, the recognizable twists remain in $\cM$,  
\item away from a very small set (a finite union of proper affine invariant submanifolds) the recognizable twists span the tangent space to $\cM$,
\item recognizable twists at the boundary of $\cM$ are also recognizable at nearby surfaces in the interior. 
\end{enumerate}

To obtain these properties, we use a version of the following result. 

\begin{thm}[Cylinder Finiteness Theorem, special case]
For each affine invariant submanifold $\cM$, the set of ratios of circumferences of $\cM$-parallel cylinders on surfaces in $\cM$ is finite. 
\end{thm}

See Section \ref{S:CF} for a more powerful version, which additionally describes moduli and cylinder deformations.

The idea of the proof of the Cylinder Finiteness Theorem is to use the $GL(2, \bR)$ action and the results of Minsky-Weiss on recurrence of horocycle flow \cite{MinW} to reduce the problem to studying cylinders of bounded circumference in a compact subset of the moduli space  of unit area surfaces. This simple plan runs up again serious technical difficulties, which we resolve using new results of independent interest on cylinder deformations. For example, we provide a partial converse to the Cylinder Deformation Theorem, for which we additionally make use of results of Smillie-Weiss \cite{SW2}. 

\begin{thm}\label{T:twistsintro}
All cylinder deformations supported on an equivalence class of $\cM$-parallel cylinders are, up to purely relative cohomology classes, equal to a scalar multiple of the standard cylinder deformation. 
\end{thm}

The standard cylinder deformation is the deformation that twists all cylinders equally; this is precisely the cylinder deformation which is guaranteed to stay in $\cM$ by the  Cylinder Deformation Theorem.

After we introduce the appropriate notation,  Theorem \ref{T:twistsintro} is restated as Theorem \ref{T:twists} in more precise language.

\begin{cor}\label{C:rational}
If the tangent space to $\cM$ contains no purely relative cohomology classes, then the moduli of $\cM$-parallel cylinders are rational multiples of each other. 
\end{cor}

The corollary can be seen as generalization of one part of the Veech dichotomy \cite{V}. More generally, one can typically obtain \ann{A: Word ``that" deleted.} more complicated rational linear relations  among moduli of $\cM$-parallel cylinders when the tangent space contains some purely relative cohomology classes. 

In the course of our study, we must prove basic results on the partial compactification $\overline{\cH}$ and how strata of small genus translation surfaces appear in the boundary of strata of large genus translation surfaces, for example Proposition \ref{P:plusv}. Previous results  primarily focused on special cases  when the difference of the genus of the boundary surface and the interior surface is at most 1  \cite{Boissy, KZ, EMZboundary}. \ann{R: When you quote [16, 28], you may also quote C. Boissy Connected components of the moduli space of meromorphic differentials. This work is related to SmillieÕs ideas on compactification.\\A: Citation added.}

\bold{Potential applications.} The results of this paper seem to greatly restrict the structure of affine invariant submanifolds. For example, they contributed significantly to the discovery, announced in early 2014, of a rank two affine invariant submanifold of $\cH(6)$ with field of definition $\bQ[\sqrt{5}]$. (At the time, the results of the present paper were conjectured to be true by the authors.) We hope the results of this paper will be similarly be useful in the future for the classification of affine invariant submanifolds, either with certain properties or in certain strata. It is conceivable that more new affine invariant submanifolds could be discovered in the process. 

The authors are currently pursuing generalizations of Theorem \ref{T:HW}. It would be interesting  to classify non-arithmetic rank 2 orbit closures in $\cH(6)$, as we have implied might be possible above. We hope that recent efforts to classify  higher rank affine invariant submanifolds in low genus might be extended using the new tools in this paper \cite{NW, ANW, AN}.

It is worth noting that Theorem \ref{T:main} is useful even without any classification of the affine invariant submanifolds that might appear in the boundary.  This theme appears already in the work of Apisa \cite{Apisa}, and will also feature in forthcoming work of the authors. 

The boundary structure of affine invariant submanifolds is of interest in the dynamics of Teichmuller geodesic flow \cite{AMY, Gadre}. The simple boundary structure of closed $GL(2,\bR)$ orbits (i.e., there are finitely many cusps) played a key role in the Veech dichotomy \cite{V}, and in the recent work of Avila and Delecroix on weak mixing in Veech surfaces \cite{AD}. \ann{A: We removed the sentence on Aulicino's joint work with Avila and Delecroix, because it is possible the set of authors for that project has changed, and we do not want to risk referring to it incorrectly.} 

 Our methods, in particular Theorem \ref{T:Sdetermines}, may clarify the extent to which the boundary of an affine invariant submanifold determines the affine invariant submanifold.

\bold{Notes and references.} For an introduction to orbit closures of translation surfaces, see, for example, \cite{Wbilliards, Wsurvey, MT, Z}. For recent results on the classification of affine invariant submanifolds, see \cite{NW, ANW, AN}. For the earlier classification results of McMullen in genus 2, see \cite{Mc5, Mc4, McM:spin}. See also \cite{Ca} for a different presentation of affine invariant submanifolds in genus 2.  For some extensions of McMullen's techniques beyond genus 2, see \cite{HLM-Q,HLM-S,N}. \ann{R: The paper by Bainbridge, Chen, Gendron, Grushevsky, Moeller: Compactification of strata of abelian differentials should be quoted by the authors (it appeared after the result of Mirzakhani and Wright). The link with this algebraic compactification should be explained.\\ A: Added citation, and some remarks after Definition \ref{D:converge} to facilitate such a comparison. Given the huge amount currently being written on compactifications, we prefer to avoid a detailed comparison here.}

A number of works, some ongoing, define compactifications of strata that remember much more information than ours, for example \cite{EKZbig, Gen, Chen, Many} and work in progress of John Smillie. Rather than deliberately forgetting the information of vanishing subsurfaces, they individually rescale the metric on each such subsurface to obtain meromorphic differentials. 
Due to the meromorphic differentials the boundary of an affine invariant submanifold in such a compactification will not be an affine invariant submanifold. 


Using disjoint methods from ours, Filip has shown that affine invariant submanifolds are quasi-projective varieties \cite{Fi1, Fi2}. It would be interesting to relate his work to the boundary theory of affine invariant submanifolds. On the one hand, one should be able to deduce some results on the boundary from Filip's work.  \ann{A: Removed two sentences.} 
Filip shows that $\cM$ can be defined by endomorphism and torsion conditions, and it would be interesting to explicitly understand the relationship between these conditions and those defining the boundary of $\cM$. This might allow a description of the boundary in algebro-geometric terms. On the other hand,  M\"oller suggested several years ago that, given sufficiently good understand of the boundary of $\cM$, one might hope to find a proof of algebraicity using Chow's Theorem \cite[Corollary 4.6]{Mum}. 

Although we  describe the boundary of an affine invariant submanifold of $\cM$, we do not describe what $\cM$ looks like in a neighborhood of the boundary. This  interesting problem has not even been addressed for strata: so far there is no explicit description of a neighborhood of a boundary stratum, except in  special cases  \cite{KZ, EMZboundary}.

\bold{Organization.} In Section \ref{S:multi} we  define our partial compactification and state our main theorem on boundaries. The proofs of basic facts about the partial compactification are deferred until the end of the paper in Section \ref{S:Basic}.  In Section \ref{S:Examples} we give examples of convergent sequences of translation surfaces. In Section \ref{S:Twist} we recall the Cylinder Deformation Theorem, and prove the partial converse. In Section \ref{S:CF} we state and prove the Cylinder Finiteness Theorem. This is the technical core of the paper. Next, in Sections  \ref{S:Rec} and \ref{S:Boundary} we define and study recognizable twists, and use them to study the boundary of affine invariant submanifolds. Most of the paper is written for the case of connected (single component) degenerations, and the general case is addressed in Section \ref{S:MultiCase}.

\bold{Acknowledgements.} We thank Matt Bainbridge, Alex Eskin, Simion Filip, Howard Masur, Curt McMullen, Kasra Rafi, Rick Schoen, Scott Wolpert, and Anton Zorich for helpful conversations. We are grateful to Curt McMullen for sharing the example in Section \ref{SS:Mc} with us.  We thank the referee for their helpful comments and suggestions.\ann{A: Added thanks to referee and Filip.}

The work of MM was partially supported by NSF and Simons grants. The work of AW was partially supported by a Clay Research Fellowship. 

\section{Multicomponent  surfaces and statement of main result}\label{S:multi}

In this section we introduce a natural class of finite area limits of translation surfaces and state our main results. We also review some definitions and some of our main tools. 

\subsection{The ``what you see is what you get" partial compactification} We begin with the definitions. 

\begin{defn}
A \emph{multicomponent translation surface} is a collection $(X, \omega, \Sigma)$. Here $X$ is a compact Riemann surface with at most finitely many connected components, $\omega$ is an Abelian differential that is nonzero on every connected component of $X$, and $\Sigma\subset X$ is a finite set of  marked points. We require that $\Sigma$ contain all the zeros of $\omega$. 
\end{defn}

 We will also consider labeled multicomponent surfaces, which are multicomponent surfaces where  the components are labelled. 
(The labeling is  to allow us refer to the different components unambiguously.)  

Every abelian differential on a connected Riemann surface defines a multicomponent translation surface as above, by marking the zeros. Marked points where $\omega$ does not vanish will be considered to be zeros of $\omega$ of order zero.  

A  partition is a set $\kappa=\{k_1, \ldots, k_s\}$, where the $k_i$ are non-negative integers (some of which may be zero). Typically we will assume $\sum k_i=2g-2$ for some integer $g$, and  will say that $\kappa$ is a partition of $2g-2$. 

 Given a finite  collection of  partitions $\kappa_1, \ldots, \kappa_p$, the stratum $\cH(\kappa_1)\times \cdots\times \cH( \kappa_p)$ of labeled multicomponent surfaces will be the collection of all labeled multicomponent translation surfaces with $p$ components, each one with zeros and marked points given by $\kappa_i$. The unlabeled stratum is a quotient of the labeled stratum, in which the labeling is forgotten. Frequently we will pass from the unlabeled situation to the labeled situation without comment by taking the pre-image under this forgetful map. 


\begin{defn}\label{D:converge}
Let $\cH'$ and $\cH$ be two strata of multicomponent translation surfaces. 
Say that $(X_n, \omega_n, \Sigma_n)\in \cH'$  converges to $(X,\omega, \Sigma)\in \cH$ if there are decreasing neighborhoods $U_n\subset X$ of $\Sigma$ with $\cap U_n = \Sigma$ such that the following holds. There are maps $g_n: X\setminus U_n \to X_n$ that are diffeomorphisms onto their range, such that 
\begin{enumerate}
\item $g_n^*(\omega_n)$ converges to $\omega$ in the compact open topology on $X\setminus \Sigma$,  
\item the injectivity radius of points not in the image of $g_n$ goes to zero uniformly in $n$. 
\end{enumerate}
Injectivity radius at a point is defined as the sup of $\e$ such that the ball of radius $\e$ in the flat metric centered at that point is embedded and does not contain any marked points.
\end{defn}

\ann{A: Added a paragraph and Remark \ref{R:partial} explicitly stating what our partial compactification is. This should make it easier to compare to other compactifications.} 
Given a stratum of connected translation surfaces, we will work in the partial compactification obtained by the union of that stratum together with all limit points of sequences of surfaces in the stratum. It follows directly from Definition \ref{D:converge} that this space admits a continuous $GL(2, \bR)$ action. 

\begin{rem}\label{R:partial}
 This partial compactification can also be obtained as follows. This description is not required for the main results in this paper. 
\begin{enumerate}
\item Given the stratum of connected translation surfaces, mark all zeros of the Abelian differentials to obtain a subset of the Hodge bundle over the moduli space of  Riemann surfaces with marked points. 
\item Take the subset of the Hodge bundle over the Deligne-Mumford compactification consisting of all finite area Abelian differentials in the closure. (This is endowed with the subspace topology.) 
\item Consider two Abelian differentials to be equivalent if they are equal after removing all zero area components (including nodes that lie on zero area components), and filling in any punctures thus created with marked points. (So if a zero area component was joined at a node to another component, that node becomes a marked point.) Take the quotient by this equivalence relation. (This quotient space is endowed with the quotient topology, which is Hausdorff.) 
\end{enumerate}
Definition \ref{D:converge} is the same as convergence in this space. This notion of convergence was also used by McMullen in \cite{McM:nav}. 
\end{rem}
So as not to distract from our main results, we defer the proof of the final claim in Remark \ref{R:partial} and the next three results to Section \ref{S:Basic}.

\begin{prop}\label{P:fn}
Suppose that $(X_n, \omega_n, \Sigma_n)\in \cH'$  converges to $(X,\omega, \Sigma)\in \cH$. Then there are ``collapse maps" $f_n$ from $X_n$ to $X$ (with some additional identifications between points of $\Sigma$)  that map marked points to marked points and have the following property: for any $\e>0$ and $n$ large enough, the $f_n$ are the inverse to $g_n$ on the subset of $(X_n, \omega_n, \Sigma_n)$ with injectivity radius at least $\e$. 
\end{prop}

Related collapse maps also appear in the theory of the Deligne-Mumford compactification, see for example \cite{Ba}.

A sequence $(X_n, \omega_n, \Sigma_n)$ may converge to a limit with more than one component. In this case one can think that some of the marked points of each component of the limit are naturally identified to some of the marked points in the other components. We do not remember these identifications in this paper, however they are required in the previous propositions, since of course there cannot be continuous maps $f_n$ from the connected $X_n$ to $X$ when $X$ is disconnected. 

In fact we will not use the $f_n$ much, only the ``space of vanishing cycles" 
$$V_n=\ker(f_n: H_1(X_n, \Sigma_n, \bC)\to H_1(X, \Sigma, \bC)).$$
Note that the relative homology group $H_1(X, \Sigma, \bC)$ is unchanged if some points of $\Sigma$ are identified to some other points of $\Sigma$, thus the map on relative homology exists even if the identifications among points of $\Sigma$ are forgotten. Although the $f_n$ are not canonically defined on the subset of with small injectivity radius, the induced map on relative homology does not depend on the choice of  map $f_n$ satisfying the conditions of Proposition \ref{P:fn}. 

\begin{prop}\label{P:Vn}
$V_n$ is well defined and eventually constant in $n$: there is a neighborhood of $(X,\omega, \Sigma)$ in $\cH'$ that contains $(X_n, \omega_n, \Sigma_n)$ for $n$ large enough on which the space of vanishing cycles is well defined and invariant under parallel transport. 
\end{prop}

In this paper we will not need the results of Rafi on the comparison between flat and hyperbolic metrics \cite{Rafi}, however those familiar with these results can anticipate the source of vanishing cycles. The  $(X_n, \omega_n, \Sigma_n)$ have a thick part with large injectivity radius, and a thin part with small injectivity radius. These two parts are separated by a collection of ``expanding annuli". The space of vanishing cycles is the relative homology of the thin part. 

Since the space of vanishing cycles $V_n$ is eventually constant, we will often omit the $n$ and write $V$ instead of $V_n$. It will be implicit that only large $n$ are considered. We will write $\Ann(V)$ for the space of cohomology classes that annihilate the vanishing cycles, using the usual pairing between relative cohomology and relative homology.  

\begin{prop}\label{P:AnnV}
In the situation above,  $\Ann(V)$  is naturally identified with the tangent space to the boundary stratum $\cH$. Furthermore, there is a neighborhood of $0$ in $\Ann(V)$ such that if $\xi_n, \xi$ are in this neighborhood and $\xi_n \to\xi,$ then  
$(X_n, \omega_n, \Sigma_n)+\xi_n$ converges to $(X, \omega, \Sigma) +\xi\in \cH$. 
\end{prop}

The above result is quite intuitive: the ways of deforming the limit surface correspond to the ways to deforming the pre-limit surfaces that do not change the parts of the surface that are degenerating.  Recall that strata are locally modeled by relative cohomology, hence our notation $+\xi$ to indicate the surface whose corresponding relative cohomology class is modified by adding $\xi$.  To prove the proposition, one must in particular show that  $(X_n, \omega_n, \Sigma_n) +\xi$ is well defined for all sufficiently small $\xi$  independent of $n$.

\subsection{Statement of results}
 
 Recall that an affine invariant submanifold is a properly immersed connected submanifold locally defined in period coordinates by real linear equations with zero constant term. We also use this definition for strata of multicomponent translation surfaces.  A foundational result on affine invariant submanifolds is the following, from \cite{EM, EMM}. 
 
 \begin{thm}[Eskin-Mirzakhani-Mohammadi]
 Any closed $GL(2,\bR)$ invariant subset of a stratum of connected translation surfaces is a finite union of affine invariant submanifolds. 
 \end{thm}
 
 The tangent bundle $T(\cM)$ of an affine invariant submanifold is naturally a subbundle of $H^1_{rel}$, the vector bundle with fibers $H^1(X, \Sigma, \bC)$. Let $H^1$ denote the vector bundle with fibers $H^1(X, \bC)$, and let $p:H^1_{rel}\to H^1$ denote the natural map.

 As the next example shows, affine invariant submanifolds may have self-crossings. 
 
\begin{ex}
Let $\cM\subset \cH(2,0,0)$ be the locus where at least one of the two marked points is at a Weierstrass point. (The marked points are not allowed to coincide, or to be equal to the zero.) Then $\cM$ self-crosses along the locus where both marked points are Weierstrass point. In one branch, one marked point stays a Weierstrass point, and in the other branch the other point stays a Weierstrass point. (McMullen has also told us a simple way to build affine invariant submanifolds with self-crossings without using marked points.)  
 \end{ex}
 
 We suppress some minor issues arising from self-crossings, whose solutions can be found in \cite[Section 2.1]{LNW}. However, we keep in mind that boundary affine invariant submanifolds may have points of self-crossing, and possibly also distinct components that cross each other.

\begin{thm}[Main Theorem]\label{T:main2}
Let $\cM$ be an affine invariant submanifold. 
\begin{enumerate}
\item The subset of the boundary of $\cM$ consisting of connected surfaces consists of a finite union of affine invariant submanifolds. If $(X_n, \omega_n)\in \cM$ converge to a connected translation surface $(X, \omega)$ in the boundary, then  $(X,\omega)$ is contained in a component  $\cM'$ of the boundary of $\cM$ whose tangent space can be computed, for infinitely many $n$, as $$T(\cM')=T_{(X_n, \omega_n)}(\cM)\cap \Ann(V).$$ Here $V=V_n$ is the space of vanishing cycles, which is eventually constant. 
\item  If $(X', \omega')$ is a point in the boundary of $\cM$, and $\pi$ is a projection onto one of its connected components, then $\pi(X', \omega')$ is contained in an affine invariant submanifold $\cM'$ whose  tangent space can be computed, for infinitely many $n$, as $$T(\cM')=\pi_*(T_{(X_n, \omega_n)}(\cM)\cap \Ann(V)).$$
\end{enumerate}
\end{thm}

The restriction to infinitely many $n$, rather than all $n$, reflects the possibility that several components of the boundary might cross at $(X,\omega)$, necessitating passing to a subsequence. Because the boundary is a \emph{finite} union of affine invariant submanifolds, Theorem \ref{T:main2}(1) implies that the sequence $(X_n, \omega_n)$ can always be decomposed into a union of finitely many disjoint subsequences, for each of which  $T_{(X_n, \omega_n)}(\cM)\cap \Ann(V)$ is eventually equal to the tangent space of a  branch of the self-crossing.  \ann{R: The restriction to infinitely many $n$... This comment is very cryptic. The end of the next paragraph (before Conjecture 2.8) is also not very illuminating for the reader.}

The first part of Theorem \ref{T:main2} is a more precise restatement  of Theorem \ref{T:main}. \ann{A: Added example and paragraph before theorem, and rewrote paragraph after.}

Given  affine invariant submanifolds $\cM_i, i=1, 2$ of strata $\cH_i$ of possibly multicomponent surfaces, one can obtain a product affine invariant submanifold $\cM_1\times \cM_2$  in the stratum of labelled multicomponent surfaces $\cH_1\times \cH_2$. If $\cM$ does not arise in this way, we say that it is irreducible. (If $\cM$ is in a stratum of unlabeled surfaces, we say it is irreducible if the connected components of its pre-image in the stratum of labelled surfaces is irreducible.)  Examples of  irreducible affine invariant submanifolds can be obtained by considering loci of multicomponent surfaces all of whose components cover a common single component translation surface.

\begin{conj}[Multicomponent EMM]\label{C:MEMM}
Let $\cM$ be a closed $GL(2,\bR)$ invariant subset of a stratum of multicomponent surfaces\ann{A: Added missing word, clarified wording.}, and assume $\cM$ consists entirely of surfaces all of whose components have equal area. Then $\cM$ is equal to the set of surfaces all of whose components have equal area in a finite union of affine invariant submanifolds. Furthermore: 
\begin{enumerate}
\item There are only countably many  closed $GL(2,\bR)$ invariant subsets of the locus where all components have equal area. 
\item For any irreducible affine invariant submanifold, locally the absolute periods of any component determine those of all other components. In particular, the ratio of areas of the components is constant. 
\item The projection of any  affine invariant submanifold to any subset of the set of components is equal to the complement of a finite union of affine invariant submanifolds in a larger affine invariant submanifold. 
\item The bundle $H^1$ over any affine invariant submanifold is semisimple. 
\end{enumerate}
\end{conj}

We believe that if any version of the main statement could be shown for multicomponent surfaces, all four following statements  would also follow immediately. The four statements can be thought of as likely corollaries of the main statement, whose proofs are deferred until the main statement has been verified.  

It is plausible that the proofs in \cite{EM, EMM} could be adapted to the multicomponent case with only very few changes. 

\begin{rem}
It is certain that at least some minor changes would be required to the proofs, because the ``diagonal" action $(g_t, g_t)$ of geodesic flow on the product of two strata is part of a $\bR^2$ action, namely $(g_s, g_t)$. Thus, even when restricting to the set where both components have equal area, the flow $(g_t, g_t)$ is not hyperbolic.
\end{rem}

\begin{thm}\label{T:main3}
 The first statement of Theorem \ref{T:main2}  holds without the connectivity assumption (i.e., multicomponent limits are allowed) if multicomponent EMM is true. 
\end{thm}

\begin{rem}
For multicomponent surfaces, not all closed $GL(2,\bR)$ invariant sets are finite unions of affine invariant submanifolds, and there are uncountably\ann{A: Typo fixed.} many affine invariant submanifolds. Both issues are created by the ability to scale the different components separately. These issues are mild, but add a few complications to the proof of Theorem \ref{T:main3}.
\end{rem} 

The body of the paper is written in the case of connected limits. Section \ref{S:MultiCase} addresses the multicomponent case, and contains the proofs of the second statement of Theorem \ref{T:main2} and the proof of Theorem \ref{T:main3}. From now until Section 8, all statements, unless we explicitly say otherwise, are for the connected case. 

\subsection{A few corollaries.}

We now turn to the proof of Corollary \ref{C:main}, which is restated in more detail below. 

Recall that the rank of an affine invariant submanifold is defined as $\frac12 \dim p(T(\cM))$ \cite{Wcyl}. (Recall $p:H^1_{rel}\to H^1$ was defined before Theorem \ref{T:main2}.)   The field of definition of $\cM$ is the smallest subfield of $\bR$ such that $\cM$ can be defined by linear equations with coefficients in this field \cite{Wfield}. 

\begin{cor}
Suppose that $\cM$ is an affine invariant submanifold, and let $\cM'$ be a component of the boundary of $\cM$. 
\begin{enumerate}
\item The dimension of $\cM'$ is less than that of $\cM$.
\item The field of definition of $\cM'$ is  equal to  that of $\cM$. 
\item The rank of $\cM'$ is at most that of $\cM$. 
\item If $\ker(p)\cap T(\cM)=\{0\}$, then the rank of $\cM'$ is strictly less than that of $\cM$. 
\end{enumerate}
\end{cor}

 The corollary is also true without projections assuming multicomponent EMM. All parts except the second are true for projections to connected components in the multicomponent case (without assuming multicomponent EMM). When projections are considered, one can only show that the field of definition of $\cM'$ is contained in that of $\cM$, and indeed there are examples that show it may be strictly smaller. 

\begin{proof}
All of these follow from Theorem \ref{T:main2}. 
\begin{enumerate}
\item This follows from $T(\cM')=T_{(X_n, \omega_n)}(\cM)\cap \Ann(V)$. The tangent space $T(\cM)$ cannot be entirely contained in $\Ann(V)$ because the $GL(2,\bR)$ directions are not contained in $\Ann(V)$. 
\item  The containment $\bk(\cM')\subset \bk(\cM)$ follows from the same formula, since $\Ann(V)$ is defined over $\bQ$. The other containment follows, for example, from \cite[Theorem 5.1]{Wfield}. Indeed,  $T_{(X_n, \omega_n)}(\cM)\cap \Ann(V)$ is defined over $\bk(\cM')$, and so in particular $T_{(X_n, \omega_n)}$ contains vectors defined over $\bk(\cM')$. By \cite[Theorem 5.1]{Wfield}, $\bk(\cM)\subset \bk(\cM')$.  
\item This follows from the same formula and the following commutative diagram, where $X'$ is $X$ with some additional identifications of points, so there are collapse maps $X_n \to X$ as in Proposition \ref{P:fn}.
 \begin{equation*}
  \xymatrix@R+2em@C+2em{
H^1(X_n, \Sigma_n) \ar[d]_-p & \ar@{_{(}->}[l] H^1(X',\Sigma') \ar[d]^-p \ar[r]  &   H^1(X,\Sigma) \ar[d]^-p  \\
  H^1(X_n)   & \ar@{_{(}->}[l] H^1(X')  \ar[r]  & H^1(X)
  }
 \end{equation*}
 As we remarked prior to Proposition \ref{P:Vn}, the top right arrow is an isomorphism. The left commutative square gives that the image of $T_{(X,\omega)}(\cM')$ in $H^1(X')$ has dimension at most twice the rank of $\cM$, and the right square gives the same statement with $X'$ replaced by $X$.

\item This follows from part 1 and the fact that the dimension is at least twice the rank with equality if $\ker(p)\cap T(\cM)=\{0\}$. 
\end{enumerate}
\end{proof}

\subsection{Degenerating cylinders.}

The following summarizes some facts that will be used throughout the paper. Say two annuli are isotopic if their core curves are isotopic. 

\begin{lem}\label{L:horrible} 
Suppose $(X_n, \omega_n, \Sigma_n)$ converges to $(X,\omega, \Sigma)$, and let $g_n: X\setminus U_n \to X_n$ be the maps given in Definition \ref{D:converge}.  The following statements hold. 
\begin{enumerate}
\item For every cylinder $C$ on $(X,\omega, \Sigma)$ and for $n$ large enough there is a corresponding cylinder $C_n$ on $(X_n, \omega_n, \Sigma_n)$   isotopic to $g_n(C)$. The height, modulus, and circumference of $C_n$ converge to those of $C$. \ann{R: What do you mean by isotopic to $g_n$?\\ A: Added reference to definition of $g_n$, and changed the wording to improve clarity.}
\item If $C_n$ are cylinders on $(X_n, \omega_n, \Sigma_n)$ of circumference uniformly bounded above and below and height uniformly bounded below, then,  possibly after passing to a subsequence, there is a cylinder $C$ on $(X,\omega, \Sigma)$ such the circumference and height and direction of $C_n$ converges to those of $C$, and such that  $C_n$ is isotopic to $g_n(C)$.  
\item In the situation of (2), for any $\e>0$, for $n$ large enough depending on $\e$, if $C_n'$ is a cylinder on $(X_n, \omega_n, \Sigma_n)$ that is parallel to $C_n$ but  not isotopic to $g_n(C')$ for a cylinder $C'$ on $(X, \omega, \Sigma)$ parallel to $C$, then either $C_n'$ has modulus less than $\e$ or circumference less than $\e$ (or both). 
\end{enumerate}
\end{lem}

In fact, all of the isotopies in this result can be assumed to be small, i.e., not to move points very much. 

The first claim is standard, and boils down to the fact that the image of a cylinder under a map which almost preserves the flat metric will contain a cylinder of similar size: a related claim is used in \cite[Proposition 4.5]{MT}, and it is also helpful to compare to the proof of \cite[Theorem A.2]{McM:nav}. The second claim is similar, and the third follows from the second.

\section{Examples of convergence to lower genus surfaces}\label{S:Examples}

In this section we give four examples of convergence according to Definition \ref{D:converge}. The first shows that the definition is easy to verify for the most commonly used degenerations. The second is actually a special case of the first, and is the simplest situation in which the degeneration cannot be obtained via local surgeries, contrary to the commonly studied situation for ``splitting a zero" and ``bubbling off a cylinder" \cite{KZ}. The last two examples  show that extremely complicated behavior is possibly along  converging sequences. 

\subsection{Degenerating cylinders.}

Let $\cC$ be a collection of parallel cylinders on a multicomponent translation surface $(X, \omega, \Sigma)$. Assume the union $\cup \cC$ of the cylinders in $\cC$ do not cover all of $(X,\omega, \Sigma)$. Let $(X_t, \omega_t, \Sigma_t), t\in (0,1)$ be the surfaces obtained by linearly scaling the height of the cylinders in $\cC$ by a factor of $t$ without shearing the cylinders. Define $(X_0, \omega_0, \Sigma_0)$ as follows. First cut out the cylinders in $\cC$, to obtain a surface with boundary. Mark every point $p$ in the boundary if it is in $\Sigma$ or if there is a vertical line in $\cup \cC$ from $p$ to a point in $\Sigma$. These finitely many marked points cut the boundary into finitely many edges. Identify two of the edges when they can be joined by vertical lines in $\cup \cC$. Make any identifications of the marked points implied by the edge identifications. 

\begin{lem}
  $(X_t, \omega_t, \Sigma_t)$ converges to the surface $(X_0, \omega_0, \Sigma_0)$. 
\end{lem}

\begin{proof}
We provide a sketch only.

\begin{figure}[h]
\includegraphics[width=\linewidth]{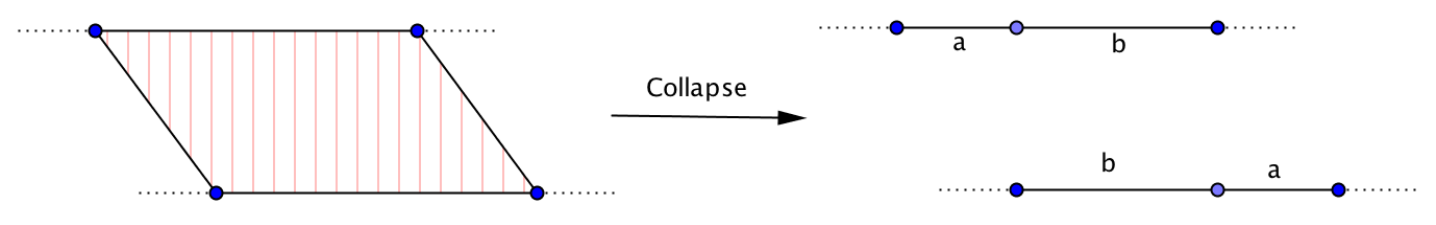}
\caption{An example of a cylinder collapse. On the left, the two non-horizontal opposite edges are identified to give a horizontal cylinder. The surface is assumed to continue above and below, but only the cylinder is drawn.  On the right, the cylinder is squashed to zero height, and becomes is a pair of saddle connections.}
\label{F:Collapse}
\end{figure}

The edges of the boundary in the construction of $(X_0, \omega_0, \Sigma_0)$ give a finite union of parallel saddle connections on $(X_0, \omega_0, \Sigma_0)$. For each of these saddle connection, we can pick a long and thin embedded rectangle centered on this saddle connection. 

\begin{figure}[h!]
\includegraphics[width=0.8\linewidth]{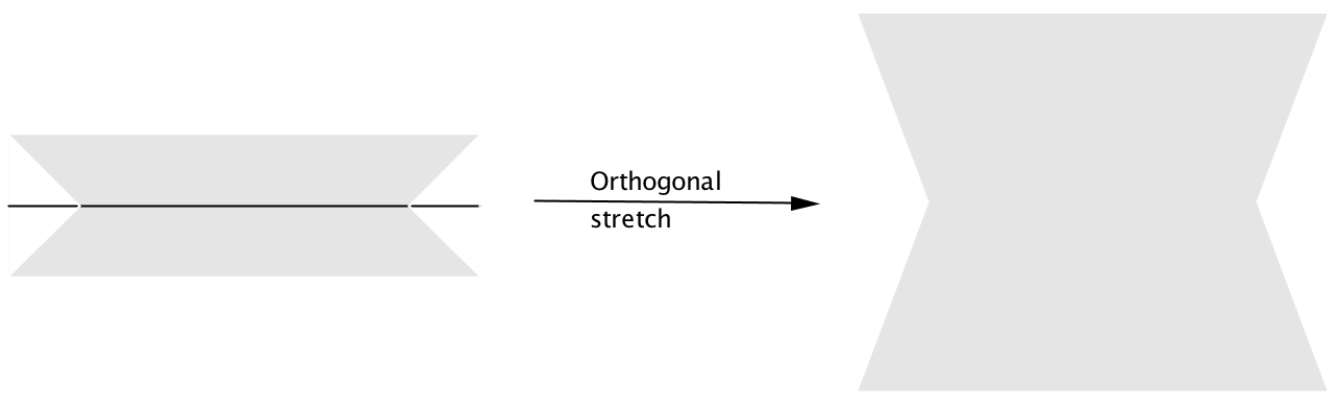}
\caption{In this example, the horizontal line is a saddle connection in the image of $\cC$ under the collapse map. The map $g_n$ takes a thin rectangle about this saddle connection and linearly stretches it in the orthogonal direction.}
\label{F:Stretch}
\end{figure}

The maps $g_n$ can be defined to be an appropriate stretch in the orthogonal direction, defined on this rectangle minus the union of small neighborhoods about the zeros. See Figure \ref{F:Stretch}. The maps $g_n$ can be taken to be the identity outside the disjoint union of these rectangles; they are not defined on the neighborhoods about the zeros.
\end{proof}

\subsection{$\cH(0,0)$ in the boundary of $\cH(2)$.}

Figure \ref{F:H2H00} illustrates a sequence of surfaces in $\cH(2)$ degenerating to $\cH(0,0)$. It is clear that the degeneration cannot be achieved by surgeries supported near the zeros, since the ratio of the periods of $\alpha$ and $\beta$  change as the surface degenerates, but $\alpha$ and $\beta$ are supported away from the zeros. 

\begin{figure}[h]
\includegraphics[width=\linewidth]{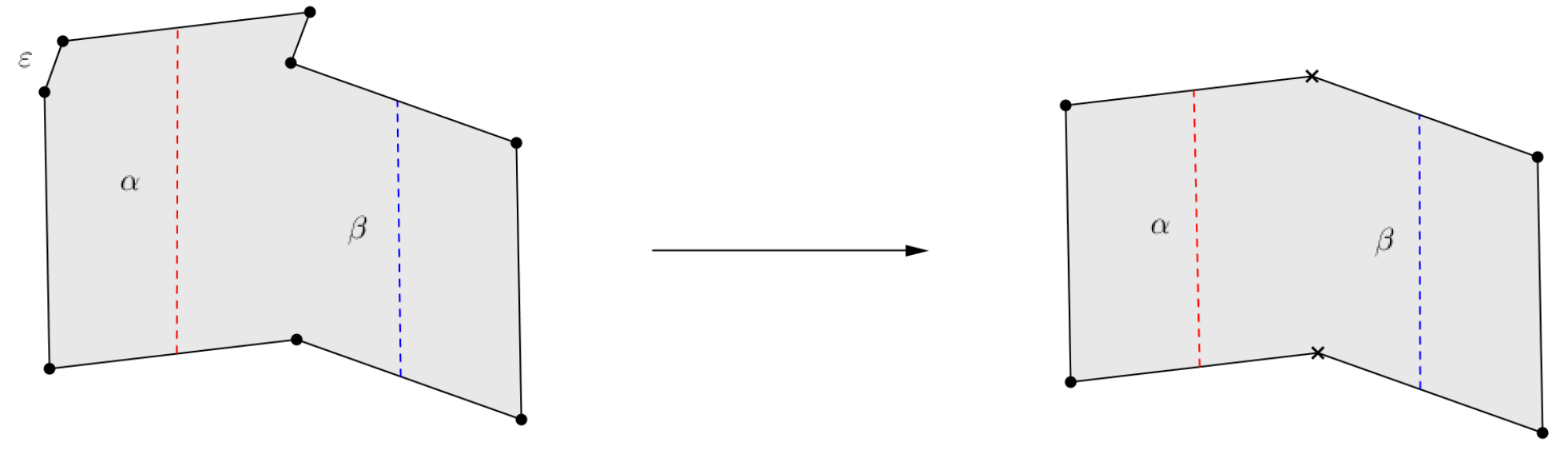}
\caption{Opposite edges are identified. The edge labeled $\e$ has length going to 0.}
\label{F:H2H00}
\end{figure}

\subsection{Chaotic behavior in the vanishing part of the surface (McMullen).}\label{SS:Mc}

This example and the next give linear paths in period coordinates that converge in the sense above, but do not converge in Deligne-Mumford. These examples are helpful but not necessary to read the rest of the paper.  

We are grateful to Curt McMullen for sharing with us this example, which shows how the chaos of the geodesic flow can be encoded in a convergent path which is even linear in period coordinates and consists entirely of unit area surfaces. 

For $t>0$ small and fixed $y$, let $E_t$ be the torus $\bC/\langle (0,1), (1-t,0)\rangle$, and let $F_t$ denote the torus $\bC/\langle (0,1), (t,y)\rangle$. Up to scaling, $F_t$ is an\ann{A: Typo fixed.} orbit of geodesic flow, and it can be arranged\footnote{Indeed, $\bC/\langle (0,1), (t,y)\rangle$, viewed as a function of $y$ for $t$ fixed, is a horosphere for the geodesic flow on the bundle of Abelian differentials over $\cM_1$. Because geodesic flow is ergodic, almost any choice of $y$ will work.} that $\{F_t: 0<t<\e\}$ is dense in $\cM_1$ for each $\e>0$. Here $\cM_1$ is the moduli space of complex tori, and we are forgetting the flat structure to get a point in $\cM_1$. \ann{R: This example is very difficult to understand. It took me
a very long time to understand the sentence It can be arranged that ...\\A: Added footnote.}

Let $(X_t, \omega_t)\in \cH(1,1)$ be $F_t$ glued to $E_t$ along a vertical slit of length $t$, as in Figure \ref{F:McM}. 
\begin{figure}[h]
\includegraphics[width=0.7\linewidth]{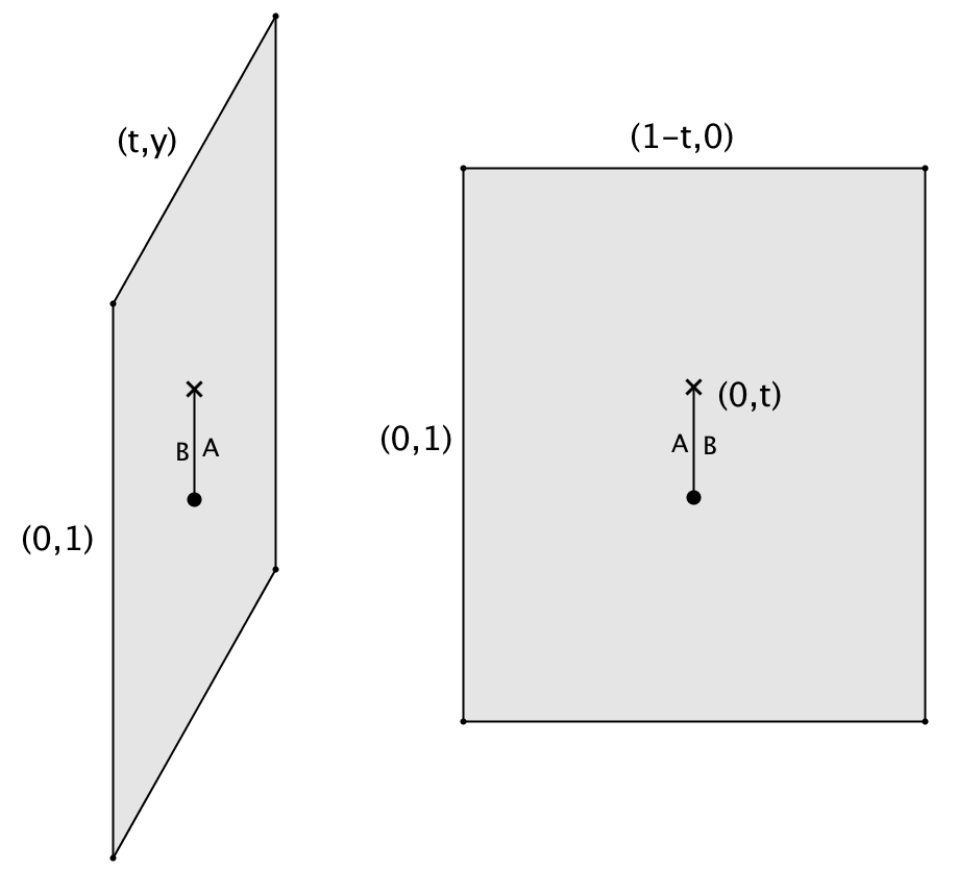}
\caption{Opposite sides are identified, so each quadrilateral gives a torus. On the left is $F_t$ and on the right is $E_t$. The edge identifications labeled A and B indicate that a slit in one torus is glued to a slit in the other.}
\label{F:McM}
\end{figure}

Observe that as $t\to 0$, the linear path $(X_t, \omega_t)$ accumulates in $\cM_2$ on $E_0\times \cM_1$. Indeed, for any $F\in \cM_1$, we may find $t$ arbitrarily small so that $F_t$ is arbitrarily close to $F$, and for these $t$ we have that $(X_t, \omega_t)$ is close to $F$ glued to $E_0$ at a point. (Note that, if we  multiply $F_t$ by $1/\sqrt{t}$ times the identify matrix, we obtain a unit area torus with a  slit of length $\sqrt{t}$. Thus the slit is very small even compared to the decreasing ``size" of $F_t$.)  

\subsection{Infinite death and rebirth of cylinders.}

The next example is not quite as chaotic as the previous one, but is more relevant in this paper. It shows the surprising phenomenon of ``infinite death and rebirth" of cylinders in a \emph{fixed homology class}, along a convergent linear path in period coordinates. This example could be easily modified to make the surfaces all of area exactly 1 (possibly changing the genus). 

Consider the surfaces $(X_t, \omega_t)$ in $\cH(2,2)$ illustrated Figure \ref{F:InfiniteRebirth}.
\begin{figure}[h!]
\includegraphics[width=\linewidth]{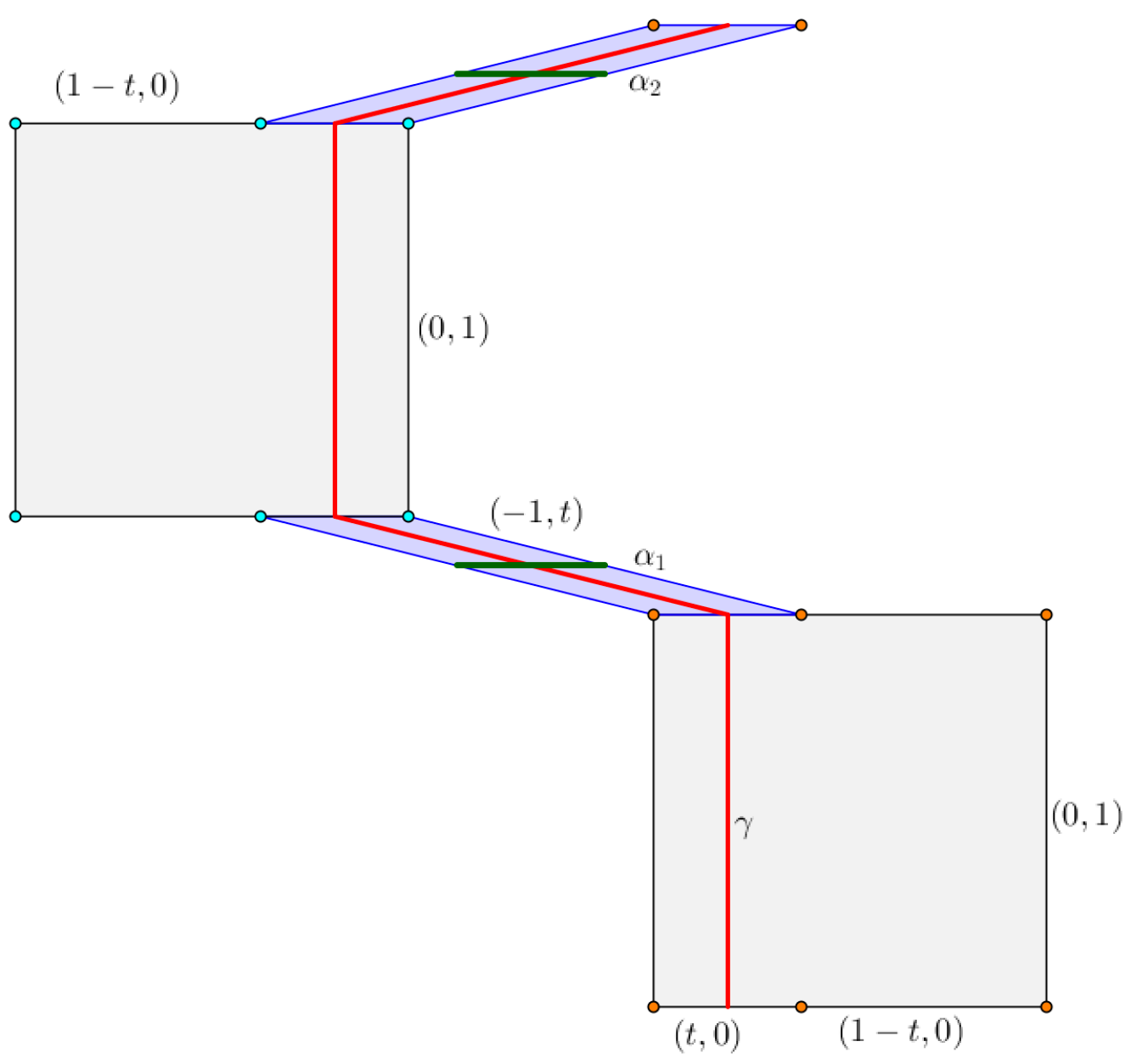}
\caption{This figure shows two tori glued together with two cylinders. The tori are both 1 by 1, i.e., $\bC/\bZ[i]$. The cylinders have height $t$. The bottom cylinder has a saddle connection going from its  bottom boundary to its top boundary with holonomy $(-1,t)$, and the top cylinder has a saddle connection going from its  bottom boundary to its top boundary with holonomy $(1,t)$. The core curves of the bottom and top cylinders are $\alpha_1$ and $\alpha_2$ respectively. }
\label{F:InfiniteRebirth}
\end{figure} 
As $t\to 0$, these surfaces converge to two tori with one marked point each. \ann{R: Why are there two marked points on each torus?\\ A: Corrected 2 to 1. }
 When $t=1/n$, there is a vertical cylinder, whose core curve is isotopic to $D^n_{\alpha_1} D^{-n}_{\alpha_2}(\gamma)$, where  $D_\alpha$ denotes the Dehn twist about a curve $\alpha$. By continuity, this cylinder also exists with $t$ is very close to $1/n$. But as $t$ gets smaller or larger, it will disappear.
 
 Since $\alpha_1$ and $\alpha_2$ are equal in homology, we see that all the core curves of these cylinders are equal in homology. 
 
In summary,  infinitely many cylinders whose core curves are homologous to $\gamma$ are ``born" and ``die" as $t\to 0$. Although the core curves of these cylinders are homologous, they are not isotopic.

\section{Cylinder twists}\label{S:Twist}

In this section we recall some background, and finally prove a result (Theorem  \ref{T:twists}) that improves previous knowledge of cylinder deformations. The reader may find it helpful to become familiar with the Cylinder Deformation Theorem, for example reading at least the introduction of \cite{Wcyl}, before proceeding. 

Note that when translation surfaces have marked points, if there is a zero inside of a cylinder, that cylinder is actually considered to be two adjacent cylinders. In other words, the interior of cylinders by definition may not contain marked points. 

\subsection{Cylinders and horocycle flow}

Suppose $(X,\omega)$ is a translation surface, and $C$ is a cylinder on $(X, \omega)$ with circumference $c$, height $h$, modulus $h/c$, and core curve $\alpha$. We assume an orientation of $\alpha$ is chosen. We will consider the cylinder twist $u_t^C(X,\omega)$, which is defined as follows. Rotate the surface by the appropriate angle $-\theta$ so that $C$ becomes horizontal. Apply 
$$u_t=\left(\begin{array}{cc} 1 & t\\0 &1\end{array}\right)$$
only to the cylinder $C$. This operation does not change the circumference of $C$ or any of the saddle connections on the boundary of $C$. Then, rotate the surface by $\theta$ to return $C$ to its original direction. The result is $u_t^C(X,\omega)$. 

In other words, $u_t^C$ shears the cylinder $C$, leaving the rest of the surface unchanged. The parameterization is such that it $u_t^C$ has period $c/h$, i.e., when $t=c/h$ then $u_t^C(X,\omega)=(X,\omega)$. 

There is a Poincare duality isomorphism $$H_1(X-\Sigma, \bZ)\simeq H_1(X, \Sigma, \bZ)^* = H^1(X, \Sigma, \bZ)$$ given by intersection number. Given a class $\alpha\in H_1(X-\Sigma, \bZ)$ we denote the dual class in $H^1(X, \Sigma, \bZ)$ by $I_\Sigma(\alpha)$.  Note that there are natural maps $$H_1(X-\Sigma, \bZ)\to H_1(X, \bZ)\quad\quad \text{and}\quad\quad p:H^1(X, \Sigma, \bZ)\to H^1(X, \bZ).$$ For a class $\alpha \in H_1(X, \bZ)$, we will denote its Poincare dual by $I(\alpha)\in H^1(X, \bZ)$. 

We will consider the core curve $\alpha$ of $C$ to be an element of both $H_1(X-\Sigma, \bZ)$ and $H_1(X, \bZ)$ without using different notation, and we have $p(I_\Sigma(\alpha))=I(\alpha)$. 

Define $u^C(X,\omega)=e^{i\theta} h I_\Sigma(\alpha)$, where $\theta$ is the argument of $\int_\alpha \omega$. We refer to $\theta$ as the direction of the cylinder. 
The derivative of the path $u_t^C(X,\omega)$ is equal to $u^C(X,\omega)$. Compare to Lemma 2.4 and Remark 2.5 in \cite{Wcyl}. 

We will say that a collection of parallel saddle connections and cylinders is consistently oriented if each saddle connection, and each core curve of a cylinder, is given an orientation so that all their holonomies are positive multiples of each other. Whenever we speak of a collection of cylinders in this paper, we will assume it is consistently oriented. 

Given a collection of cylinders $\cC=\{C_1, \ldots, C_k\}$, define $u^\cC_t(X,\omega)=u_t^{C_1} \cdots u_t^{C_k} (X,\omega)$ to be the result of simultaneously shearing each $C_i$. Suppose $C_i$ has height $h_i$, and define $u^\cC=e^{i\theta}\sum h_i I_\Sigma(\alpha_i)$, where $\theta$ is the direction of the cylinders. The derivative of the path $u_t^\cC(X,\omega)$ is $u^\cC(X,\omega)$.  Compare to Lemma 2.4 and Remark 2.5 in \cite{Wcyl}.

Let $\cM$ be an affine invariant submanifold. Recall that if $(X, \omega)\in \cM$, then two cylinders on $(X,\omega)$ are said to be $\cM$-parallel if they are parallel, and if they remain parallel on all sufficiently nearby surface in $\cM$ \cite[Definition 4.6]{Wcyl}. (More precisely: There is a small simply connected neighborhood of $(X,\omega)$ on which both cylinders persist, and whose intersection with $\cM$ is connected. We require the two cylinders to be parallel on all surfaces in this neighborhood intersect $\cM$. See \cite[Section 2.1]{LNW} for the case when $\cM$ has self-crossings.\ann{A: Added citation.}) 

If two cylinders are $\cM$-parallel, then the ratio of their circumferences must  remain constant on all nearby deformations \cite[Lemma 4.7]{Wcyl}. The following is \cite[Theorem 5.1]{Wcyl}. 

\begin{thm}[Cylinder Deformation Theorem]\label{T:CDT}
Let $\cM$ be an affine invariant submanifold. 
If $\cC$ is an  equivalence class of $\cM$-parallel cylinders on $(X, \omega)\in \cM$, then $u_t^\cC(X, \omega)$ stays in $\cM$ for all $t\in \bR$. Equivalently, $\sum h_i I_\Sigma(\alpha_i)\in T(\cM)$. 
\end{thm}

The proof of Theorem \ref{T:CDT} uses the following result. 

\begin{thm}[Smillie-Weiss]\label{T:SW}
Every closed, horocycle flow invariant subset of a stratum contains a horizontally periodic surface. 
\end{thm}

\subsection{Cautionary examples.}

The possibility that small deformations can create new cylinders $\cM$-parallel to the existing cylinders is a major complicating difficulty in this paper, so we encourage the reader to consider the following two examples in detail.

\begin{ex}
When $\cM$ is a stratum, two cylinders are $\cM$-parallel if and only if they are homologous.
Take $\cM$ to be a stratum that contains a surface with a pair of homologous cylinders. Consider a surface in this stratum with a pair of homologous cylinders. Keeping the circumference constant, collapse one cylinder so that it has zero height, without degenerating the surface. Let the resulting surface be $(X, \omega)$, and let $\cC$ be the one remaining cylinder from the homologous pair. Then, for many deformations, such as the reverse of the path we just described, a new cylinder will appear homologous to the one cylinder in $\cC$.  
\end{ex}

\begin{ex} Take $\cM$ to be an eigenform locus in $\cH(1,1)$ \cite{Mc}, see also \cite{Ca}. Thus $\cM$ is rank 1, and so all parallel cylinders are automatically $\cM$-parallel, see \cite[Propositions 4.19, 4.20]{Wsurvey} or \cite[Theorem 1.5]{Wcyl}.   Consider a surface $(X, \omega)\in \cM$ with a two cylinder direction. Let $\cC$ be these two cylinders. Most small deformations of $(X, \omega)$ in $\cM$ result in a surface where there are three cylinders in the direction of the cylinders in $\cC$.  

Furthermore, McMullen has established the following result \cite{McCas,McM:iso}. 
In each eigenform locus $\cM$ in genus 2, there are linear paths $(X_t, \omega_t)$ in period coordinates that converge as $t\to 0$ to a surface in the eigenform locus (not its boundary) with the following properties. For some $\e>0$ and all $t\in (0,\e)$, the surfaces $(X_t, \omega_t)$ are all horizontally periodic. Furthermore, absolute periods are constant along the path $(X_t, \omega_t)$, and  as $t\to 0$, the circumference of the smallest horizontal cylinder goes to infinity.
  
Since absolute periods are constant along the path $(X_t, \omega_t)$, the circumference of cylinders are locally constant. In order for the circumference of the smallest horizontal cylinder to go to infinity as $t\to 0$, the horizontal cylinders must disappear and be replaced with longer and longer horizontal cylinders. Infinitely many horizontal cylinders must appear and then die off. \ann{R: Since the cylinders are getting longer and longer ... I do not understand this argument.\\A: Rephrased.} The ``lifespan" of each horizontal cylinder goes to zero as $t\to 0$. Furthermore, since these $\cM$ are rank 1, all the core curves of the horizontal cylinders in this family are $\cM$ collinear. 
  
This fact that $\cM$-parallel cylinders in a given direction may appear and disappear infinitely many times along finite paths is a major source of complication in this paper. 

See \cite{HW} for related examples, with a very explicit analysis of how infinitely many cylinders appear and disappear along a small path $(X_t, \omega_t), \,0< t<\e$ . (In most of the examples in \cite{HW}, the cylinders are not $\cM$-parallel.)
\end{ex}

\subsection{Non-standard twists}

Given an equivalence class $\cC$ of $\cM$-parallel cylinders with core curves $\alpha_i$, define its \emph{twist space} to be  
$$(\span_\bC I_\Sigma(\alpha_i)) \cap T(\cM).$$ This space describes all deformations that remain in $\cM$ that arise from deforming cylinders in $\cC$. 

We will call $u^\cC$ the standard twist in $\cC$. By the Cylinder Deformation Theorem, it is always contained in the twist space.   

\begin{war} The definition of the standard twist  depends on $(X, \omega)$. For example, suppose $(X', \omega')\in \cM$ is close enough to $(X, \omega)$ so that all cylinders in $\cC$ persist at $(X', \omega')$. It may be that the heights of the cylinders in $\cC$ at $(X, \omega)$ and $(X', \omega')$ are not equal. Worse yet, it is possible that at $(X', \omega')$ there are cylinders $\cM$-parallel to those in $\cC$ but not contained in $\cC$.  Either of these possibilities show that the standard twist in an equivalence class of $\cM$--parallel cylinders is not locally constant, even when the cylinders persist. 
\end{war}

\begin{lem}\label{L:rational}
Suppose that $(X, \omega)\in \cM$. Let $\cC$ be the collection of all cylinders in $(X, \omega)$ in some direction, so in particular all cylinders in $\cC$ are parallel. Let $\cC=\{C_1, \ldots, C_k\}$ and suppose that $C_i$ has circumference $c_i$, modulus $m_i$, height $h_i$, and core curve $\alpha_i$. Then the subspace of 
$$\left\{\sum t_i c_i I_\Sigma(\alpha_i): (t_1, \ldots, t_k)\in \bC^k\right\}$$
that is contained in $T(\cM)$ is defined by linear equations on the variables $t_i$ with coefficients in $\bQ$ that are satisfied when all $t_i=m_i$. 
\end{lem}

In this lemma, and for the  remainder of the paper, we choose to parameterize cylinder deformations as linear combinations of $c_i I_\Sigma(\alpha_i)$ rather than $ I_\Sigma(\alpha_i)$. Without this scaling of the coordinates, the rationality assertion (coefficients in $\bQ$) of this lemma would be false. 

\begin{proof}
This follows from a standard argument. Namely, the collection of all surfaces obtained by twisting the cylinders arbitrary amounts forms a torus of dimension $k$ in the stratum. All twists that stay in $\cM$ give rise to a subtorus, which must be defined by rational equations in the coordinates. Since the $t_i=e^{i\theta} m_i$ is the standard twist and stays in $\cM$ by the Cylinder Deformation Theorem, all these rational equations must be satisfied at $t_i=m_i$. See  \cite[Corollary 3.4]{Wcyl} for details.
\end{proof}

\begin{thm}\label{T:twists}
The twist space of an equivalence class of $\cM$-parallel cylinders is contained in the span of the standard twist and $\ker(p)\cap T(\cM)$. 
\end{thm}

The proof will be by contradiction, using the following lemma. 

\begin{lem}\label{L:Lag}
Suppose that $(X, \omega)\in\cM$ has an equivalence class of $\cM$-parallel horizontal cylinders that supports two twists $u, u'$ such that $p(u)$ and $p(u')$ are not collinear. 

Then there is an $(X', \omega')\in \cM$ with the same property, such  that $(X', \omega')$ is horizontally periodic and has a collection $\cC_1, \ldots, \cC_r$ of equivalence classes of $\cM$-parallel horizontal cylinders  such that $p(u^{\cC_i}), i=1, \ldots, r$ is a basis of a Lagrangian in $p(T(\cM))$.  We may also assume that  $u$ and $u'$ are supported on $\cC_1$.
\end{lem}

\ann{A: Removed a sentence.}
Here, by definition, $r$ is the rank of $\cM$.

\begin{proof}
This follows by a non-trivial argument that uses Theorem \ref{T:SW} repeatedly. The exact argument required appears in \cite[Section 8]{Wcyl}.
\end{proof}

\begin{proof}[Proof of Proposition.]
Suppose in order to find a contradiction that $(X, \omega)$ has an equivalence class of $\cM$-parallel horizontal cylinders that supports two twists $u, u'$ such that $p(u)$ and $p(u')$ are not collinear. Let $(X', \omega')$ be as in the previous lemma. 

Note that $p(u), p(u')$ together with  $p(u^{\cC_i}), i=1, \ldots, r$ span an isotropic subspace. Since $p(u^{\cC_i}), i=1, \ldots, r$ are a basis for a Lagrangian, $p(u), p(u')$ can be extended to a basis for this Lagrangian by adding some of the $p(u^{\cC_i})$. Thus, without loss of generality, we may assume that  $p(u^{\cC_i}), i=1, \ldots, r-1$ together with $u, u'$  span this Lagrangian. 

This would imply that $p(u^{\cC_r})$ can be written as a linear combination of $p(u), p(u')$  and $p(u^{\cC_i}), i=1, \ldots, r-1$ with \emph{real} coefficients. (The coefficients can be assumed to be real because all these vectors are contained in the real vector space $H^1(X, \Sigma, \bR)$.)

By definition, each $p(u^{\cC_i})$ is dual to a real positive linear combination of core curves of cylinders in $\cC_i$. Similarly, both $u$ and $u'$ are dual to a real linear combination of core curves of cylinders in $\cC_1$. Hence, we conclude that some positive linear combination of core curves of cylinders in $\cC_r$ is homologous to a real linear combination of core curves of the $\cC_i, i=1, \ldots, r-1$. 

However, the fact that the $p(u^{\cC_i}), i=1, \ldots, r$ are linearly independent implies that the $\cC_r$ can be made not horizontal on a nearby surface $(X'', \omega'')\in \cM$ while keeping the $\cC_i, i=1, \ldots, r-1$ horizontal.  This contradicts the homology relation in the previous paragraph. 
 \end{proof}

\section{Cylinder finiteness}\label{S:CF}

This section is the technical heart of this paper. Its\ann{A: Typo fixed.} purpose is to show the following. 

\begin{thm}[Cylinder Finiteness Theorem]\label{T:CFT}
Fix an affine invariant submanifold $\cM$.
There are two {finite} sets $S_1\subset \bk(\cM)$ and $S_2\subset \bQ$, such that if $(X,\omega)\in \cM$, and $\cC$ is an equivalence class of $\cM$-parallel cylinders on $(X, \omega)$, then the ratio of circumferences of any two cylinders  in $\cC$ is in $S_1$. 

Furthermore, if $C_i$ has circumference $c_i$ and core curve $\alpha_i$, then the twist space for $\cC$ at $(X,\omega)$ is a subspace of 
$$\left\{\sum t_i c_i I_\Sigma(\alpha_i)\right\}$$ 
defined by linear equations on the $t_i$ with coefficients in $S_2$ that are satisfied when all $t_i=m_i$ are the moduli of the cylinders $C_i$. 
\end{thm}

Besides the statement of Theorem \ref{T:CFT}, no results from this section will be used in the rest of the paper. The strength of the statement is that $S_1$ and $S_2$ are finite. Indeed, it is easy to see that the ratio of two $\cM$-parallel cylinders must be contained in $\bk(\cM)$, see \cite[Section 7]{Wcyl}, and we saw in Lemma  \ref{L:rational} that twist spaces are always defined by  linear equations on the $t_i$ with rational coefficients. 

\begin{cor}\label{C:rational2}
There is a finite set $S_2'\subset \bQ$ depending only on $S_2$ and the maximum size of an equivalence class of $\cM$-parallel cylinders such that the following is true. 

Suppose that $\cC$ is an equivalence class of $\cM$-parallel cylinders on $(X,\omega) \in \cM$, such that the ratio of moduli of cylinders in $\cC$ cannot be changed by deforming the surface in $\cM$.  Then the ratio of moduli of cylinders in $\cC$ is in $S_2'$.
\end{cor}

By Theorem \ref{T:twists}, the assumption on $\cC$ is always satisfied\ann{R: typo: always the satified\\A: Corrected.} if $\ker(p)\cap T(\cM)=\{0\}$, and hence Corollary \ref{C:rational} is a special case of Corollary \ref{C:rational2}. A trivial bound for the maximum size of an equivalence class of $\cM$-parallel cylinders is given by the maximum size of a set of parallel cylinders in the stratum.

\begin{proof}[Proof of Corollary.]
The twist space is defined by equations on the $t_i$ with coefficients in $S_2$, and the only solution is $t_i=m_i$. There are only finitely many such systems of equations, and we may take $S_2'$ to be the set of all ratios of $t_i$ for the solutions of all these systems. 
\end{proof}

The strategy of the proof of Theorem \ref{T:CFT} is to use the following corollary of the quantitative recurrence results of Minsky-Weiss \cite{MinW}. For any affine invariant submanifold $\cM$, let $\cM^1$ denote the set of unit area surfaces in $\cM$.

\begin{thm}[Minsky-Weiss]\label{T:MW}
Fix a stratum $\cH$. For every $\e>0$, there is a compact subset $K\subset \cH^1$, such that if $(X, \omega)$ has no horizontal saddle connections of length less than $\e$, then the horocycle flow orbit of $(X,\omega)$ intersects $K$. 
\end{thm}

We wish to use this result to show that any equivalence class of $\cM$-parallel cylinders can be ``moved" to a fixed compact set in a way that the smallest cylinder becomes circumference 1, and then use a compactness argument to obtain finiteness. Both parts of this strategy encounter major technical difficulties. These are handled in the next two subsections, after which Theorem \ref{T:CFT} follows easily using Theorem \ref{T:MW}.

We will make repeated use of the following matrices: 
$$u_t=\left(\begin{array}{cc} 1 & t\\0 &1\end{array}\right) \quad\quad\text{and}\quad\quad 
g_t=\left(\begin{array}{cc} e^t & 0\\0 &e^{-t}\end{array}\right).$$

\subsection{Saddle connections parallel to cylinders.} 

Recall that a cylinder is said to be $\cM$-parallel to a saddle connection if, locally on $\cM$, the holonomy of the saddle connection is a fixed real multiple of the holonomy of the core curve of the cylinder. 

\begin{prop}\label{P:notinySC}
Fix an affine invariant submanifold $\cM$. Then there is a positive constant $R_{sc}$ such that the follow holds. For every $(X,\omega)\in \cM$, and for every saddle connection $w$ on $(X,\omega)$, if there is a cylinder $\cM$-parallel to $w$, then there is cylinder $\cM$-parallel to $w$ whose circumference is at most $R_{sc}$ times the length of $w$ and whose area is at least $1/R_{sc}$ times the sum of the areas of the cylinders $\cM$-parallel to $w$. 
\end{prop}

In the next lemma, we will consider positive linear combinations of parallel consistently oriented saddle connections. Recall that consistently oriented means that the holonomies are positive multiples of each other. The key  property that such linear combinations have in common with individual saddle connections is that if the holonomy does not change under a standard cylinder deformation,  then the saddle connections are disjoint from the cylinders. We consider a saddle connection and a cylinder to be disjoint if the interior of a saddle connection is disjoint from the interior of the cylinder, so for example a saddle connection on the boundary of a cylinder is considered to be disjoint from the cylinder.  
\ann{R: I have a stupid question about the definition of cylinders: do you include the boundary or not? It is important when you build cylinders disjoint from w.\\ A: Clarified.}

Recall that two relative homology classes $\gamma$ and $w$ are called $\cM$-collinear if, for some $c\in \bR$, the equation $\int_\gamma \omega=c\int_w \omega$ is one of the linear equations locally defining $\cM$ in period coordinates. If, as in the next lemma, $\gamma$ is an absolute homology class but $w$ is not, this means that $w$ will behave like an absolute homology class for surfaces in $\cM$. 

\begin{lem}\label{L:still disjoint}
Suppose that $\cC$ is an equivalence class of $\cM$-parallel cylinders on $(X,\omega)\in \cM$.  Let $w$ be a linear combination of saddle connections on $(X,\omega)$ that are disjoint from the cylinders in $\cC$, and assume that $w$ is $\cM$-collinear to an absolute homology class $\gamma$.   

 Then for $(X', \omega')\in \cM$ sufficiently close to $(X,\omega)$, if $\cC'$ is the equivalence class of cylinders $\cM$-parallel to cylinders in $\cC$,  and $w'$ is a positive linear combination of consistently oriented parallel saddle connections on $(X', \omega')$ equal to $w$ in relative homology, then the  cylinders in $\cC'$ are disjoint from $w'$ on $(X', \omega')$. 
\end{lem}

Let $\cU$ be a neighborhood of $(X,\omega)$ in $\cM$, small enough so that all cylinders in $\cC$, persist on all surfaces in $\cU$. In the lemma, ``close enough" means $(X', \omega')\in \cU$.  Note that $\cC'$ may be strictly larger than $\cC$.

\begin{proof}
We assume $\gamma$ is chosen so that $\int_\gamma \omega=\int_w \omega$. 

Let $u=u^\cC(X,\omega)$ be the standard cylinder twist in the cylinders of $\cC$ at $(X, \omega)$, and let $u'=u^{\cC'}(X', \omega')$ be the standard twist in the cylinders in $\cC'$ at $(X', \omega')$. 

By assumption, $u(w)=0$.  Theorem \ref{T:twists} gives that $p(u)$ and $p(u')$ are collinear. Hence $u'(w)=u'(\gamma)$ is proportional to $u(\gamma)=u(w)=0$, so we get that $u'(w)=u'(w')=0$. The result follows. 
\end{proof}

\begin{lem}\label{L:disjointcyls}
Fix an affine invariant submanifold $\cM$, and suppose that $(X,\omega)\in \cM^1$ and that $w$ is a positive linear combination of consistently oriented horizontal saddle connections on $(X,\omega)$. Suppose that $w$ is $\cM$-collinear to an absolute homology class. 

Then there is a constant $\delta>0$ depending only on the stratum such that there is an equivalence class $\cC$ of $\cM$-parallel cylinders on $(X,\omega)$ that is disjoint from $w$ and that has area at least $\delta$.
\end{lem}

\begin{proof}
Set $\delta$ to be 1 divided by twice the maximum number of parallel cylinders on a surface in the given stratum. Let $(X', \omega')\in \cM$ be a horizontally periodic surface in the horocycle flow orbit closure of $(X,\omega)$. 

Let $\cC$ be an equivalence class of $\cM$-parallel horizontal cylinders on $(X', \omega')$ that occupy at least $2\delta$ of the area. 

Pick $t$ so that $u_t(X, \omega)$ is  close enough to $(X', \omega')$ so all cylinders in $\cC$ persist at $u_t(X, \omega)$, and so that they occupy at least $\delta$ of the area. 

Let $\cC'$ be the equivalence class of cylinders on $u_t(X, \omega)$ that are $\cM$-parallel to cylinders in $\cC$. By  Lemma \ref{L:still disjoint}, the cylinders in $\cC'$ are disjoint from $w$. 

Pulling the cylinders of $\cC'$ back to $(X,\omega)$ using $u_{-t}$ gives the result. 
\end{proof}

\begin{proof}[\eb{Proof of Proposition \ref{P:notinySC}.}]
Assume to the contrary that there is a sequence $(X_n, \omega_n)\in \cM^1$ of unit area surfaces where there is a saddle connection $w_n$ that is $\cM$-parallel to a cylinder, but that the set $\cC_n$ of such cylinders can be broken up into two disjoint subsets $\cC_n=\cC_n'\cup \cC_n''$ such that all cylinders in $\cC_n'$ have circumference, divided by the length of $w_n$, going to infinity uniformly in $n$, and the cylinders in $\cC_n''$ have  area, divided by the total area of $\cC_n$, going to zero uniformly in $n$. 

Slightly deforming the surfaces, we may assume that $w_n$ is not parallel to any saddle connection that it is not $\cM$-parallel to. (Any cylinders created in this process that are $\cM$-parallel to those in $\cC_n$ can be assumed to have arbitrarily small area.) Without loss of generality, we will assume that any saddle connection parallel to $w_n$ is at least as long as $w_n$.  

Rotating the surfaces, we may assume that all $w_n$ are horizontal. The reason we choose to have $w_n$  horizontal is so that $w_n$ and $\cC_n$ will be preserved by the horocycle flow. Applying $g_t$, we may assume that $w_n$ has length 1. By applying the cylinder stretch in $\cC_n$ (the deformation with derivative $i u^{\cC_n}$, which increases heights of cylinders) we can increase the ratio of the area of the cylinders in $\cC_n$ to the total area. This changes the area of the surface, but we can then apply matrices of the form 
$$\left(\begin{array}{cc} 1 & 0\\0 &a\end{array}\right)$$
to renormalize the area. The result is that we may assume that the area of $\cC_n$ goes to 1 as $n$ goes to infinity. Hence the area of $\cC_n'$ goes to 1 as $n$ goes to infinity. 

By applying horocycle flow and using Theorem \ref{T:MW}, we may assume that all the $(X_n, \omega_n)$ lie in some fixed compact set. Hence, by passing to a subsequence, we may assume that they converge to a limit in the same stratum. Say the limit is $(X,\omega)$. 

This limit has a finite union of horizontal saddle connections of length 1, which we will call $w$. (Multiplicity is allowed.) This is the limit of $w_n$ on $(X_n, \omega_n)$, and has the same relative homology class as the $w_n$ when $n$ is large enough.
The relative homology class of $w$ is, by assumption, $\cM$-collinear to any $\gamma$ which is the core curve of a cylinder in $\cC_n$ on $(X_n, \omega_n)$. 

Hence by Lemma \ref{L:disjointcyls}, we get the existence of an equivalence class of $\cM$-parallel cylinders $\cD$ on $(X,\omega)$ that is disjoint from $w$, and occupies area at least $\delta$. Pick $n$ large enough so that $\cC_n'$ has area greater than $1-\delta/2$, and such that the cylinders of $\cD$ persist on $(X_n, \omega_n)$ and occupy area at least $\delta/2$. 

Let $\cD_n$ be the set of cylinders on $(X_n, \omega_n)$ that are $\cM$-parallel to the cylinders in $\cD$. Lemma \ref{L:still disjoint} shows that $\cD_n$ is disjoint from $w_n$. Furthermore, $\cD_n$  contains cylinders whose circumference and area are uniformly bounded above and below, namely those coming from $\cD$. Hence  $\cD_n$ is not contained in $\cC_n$, and hence no cylinder in $\cD_n$ is contained in $\cC_n$ and vice versa. (Since they are equivalence classes, if $\cC_n$ and $\cD_n$ had a single cylinder in common, they would be equal.) 

However, by area considerations, some cylinder in $\cD_n$ must intersect some cylinder in $\cC_n$. This is a contradiction, because the cylinder deformation in the $\cD_n$ changes the circumferences of at least one cylinder in $\cC_n$, without changing the holonomy of $w_n$, contradicting the fact that $w_n$ and the cylinders of $\cC_n$ are $\cM$-parallel. 
\end{proof}

\subsection{Bounding cylinder ratios.}

\begin{prop}\label{P:Cbound}
Fix an affine invariant submanifold $\cM$. Then there is a constant $R_{cyl}$ such that  any  ratio of circumferences of any pair of $\cM$-parallel cylinders on any surface in $\cM$ is at most $R_{cyl}$. 
\end{prop}

The following basic lemma will be helpful in combination with Theorem  \ref{T:twists} to understand how new cylinders can ``appear" in equivalence classes as the surface is deformed. 

\begin{lem}\label{L:hom}
Suppose that $\alpha_i, i=1, \ldots, n$ and $\beta_j, j=1, \ldots, m$ are all core curves of a set of consistently oriented distinct parallel cylinders on a translation surface $(X, \omega)$.  Suppose furthermore that in homology there is some linear relation 
$$\sum a_i \alpha_i = \sum a_i' \alpha_i+ \sum b_j \beta_j$$
with all $a_i$ and $a_i'$ and $b_j$ positive real numbers. 

Then for each $j_0$ there is a subset $J\subset \{1, \ldots, m\}$ containing $j_0$ and two  subsets $I, I'\subset \{1, \ldots, n\}$ such that in homology, 
$$\sum_{i\in I} \alpha_i = \sum_{i\in I'} \alpha_i +\sum_{j\in J} \beta_j.$$

In particular, for each $j$ we have $$\left|\int_{\beta_j} \omega\right| \leq  \sum_{i=1}^n \left|\int_{\alpha_i} \omega\right|.$$
\end{lem}

Recall that consistently oriented means that the $\int_{\alpha_i}\omega, \int_{\beta_j}\omega$ are all positive multiples of each other.

\begin{proof}
Without loss of generality, assume the cylinders are horizontal. 
Cut out all the $\alpha_i$ and $\beta_j$ from $X$, to obtain a finite union $X_1, \ldots, X_k$ of connected components, each of which caries a flat structure and has horizontal boundary.

Suppose that $\beta_{j_0}$ is part of the top boundary of $X_1$. Iteratively perform the following operation, illustrated in Figure \ref{F:hom}. Start with $X_1$. If there is a bottom boundary component that is a $\beta_j$, glue on the piece $X_p$ that contains this $\beta_j$ in the top. In the object thus obtained, if there is a bottom boundary component that is a different $\beta_j$, glue on the piece $X_{p'}$ that contains this $\beta_j$ in the top. 

\begin{figure}[h]
\includegraphics[width=0.5\linewidth]{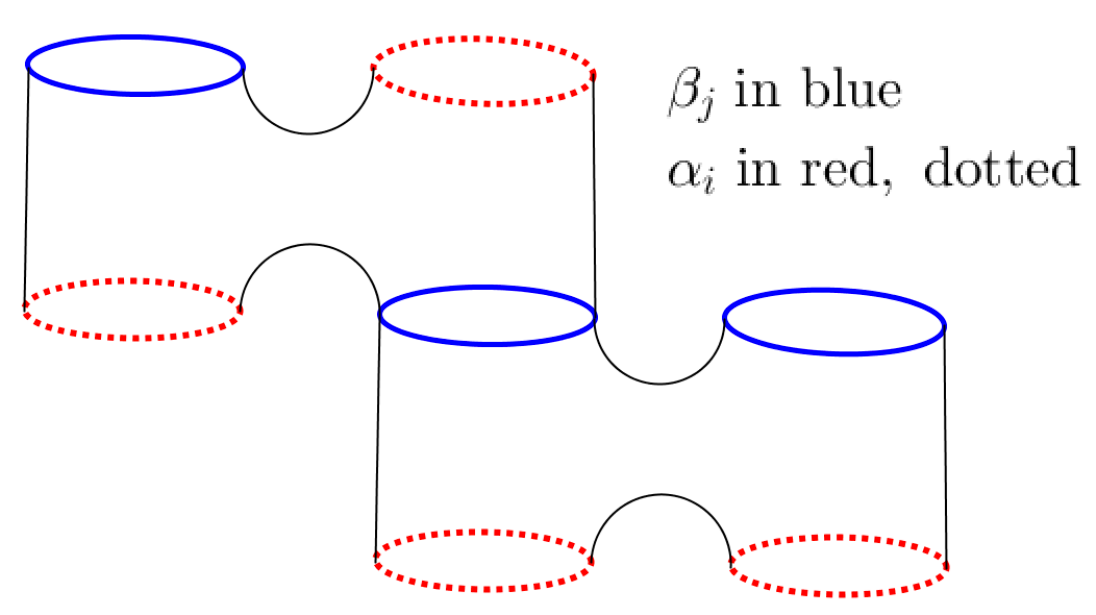}
\caption{Proof of Lemma \ref{L:hom}.}
\label{F:hom}
\end{figure}

We claim that when this process is completed, $\beta_{j_0}$ is contained in the top boundary of the resulting object (and not in the bottom--there are no $\beta_j$ in the bottom boundary). Indeed, otherwise we could find a curve $\gamma$ going through the pieces and crossing only the $\beta_j$ curves and not crossing any $\alpha_i$ curves. This curve would contradict the assumed linear relation. 

The lemma is now proved, with $J$ the set of $j$ such that $\beta_j$  occurs in the top boundary, $I$ the set of $i$ such that $\alpha_i$  occurs in the bottom boundary, and $I'$ the set of $i$ such that $\alpha_i$  occurs in the  top  boundary.  

\end{proof}

\begin{lem}\label{L:atmostsum}
Let $\cM$ be an affine invariant submanifold, and $(X,\omega)\in \cM$. Let $\cC=\{C_1, \ldots, C_n\}$ be an equivalence class of $\cM$-parallel cylinders on $(X,\omega)$. Suppose that $(X', \omega')\in \cM$ is a small deformation of $(X,\omega)$, so all the $C_i$ persist on $(X', \omega')$. 
Suppose that $C_0$ is a cylinder on $(X', \omega')$ that is $\cM$-parallel to the $C_i, i=1, \ldots, n$. Then the circumference of $C_0$ is at most the sum of the circumferences of the core curves of the cylinders in $\cC$. 
\end{lem}

\begin{ex}
It is however possible that the core curve of $C_0$ is not in the span of the core curves of the $C_i, i=1, \ldots, n$. In Figure \ref{F:TwoNewCyls} we draw an example of a locus $\cM$ of degree 2 torus covers in $\cH(1,1)$. The surface on the right is the $(X,\omega)$. Many deformations, such as those on the left, have two ``new" cylinders that are $\cM$-parallel to the  cylinder that persists from $(X,\omega)$. However, these cylinders are not homologous to the cylinder that persists from $(X,\omega)$.
\begin{figure}[h]
\includegraphics[width=0.9\linewidth]{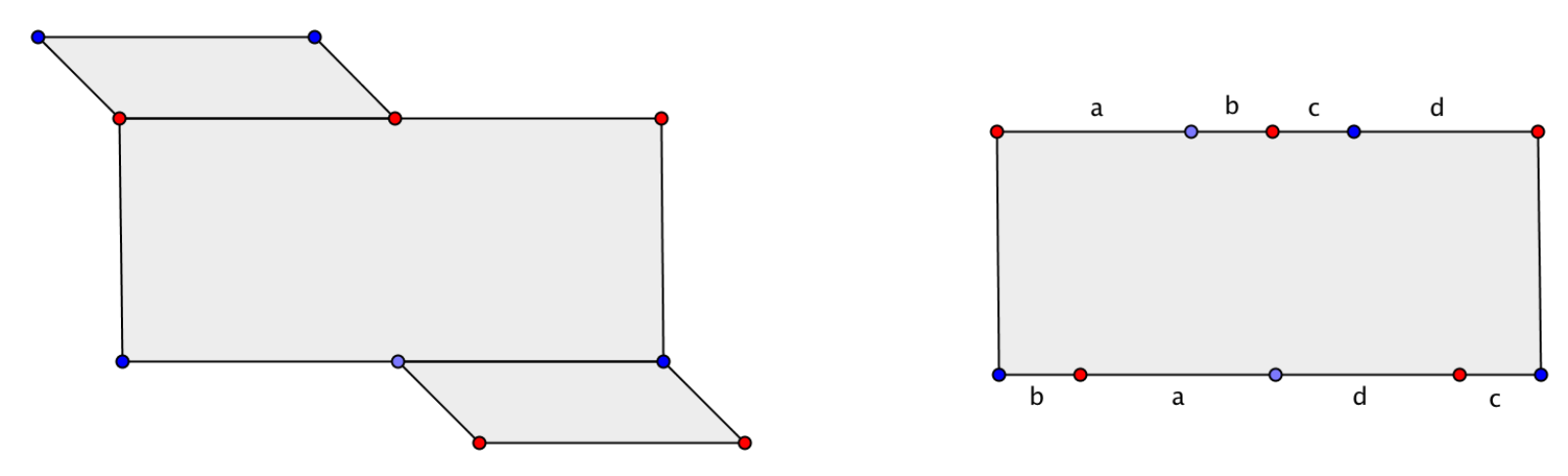}
\caption{}
\label{F:TwoNewCyls}
\end{figure}
\end{ex}

\begin{proof}
Let $u$ be the standard twist in $\cC$ at $(X,\omega)$.  Let $\cC'$ be the equivalence class of $C_1$ at $(X',\omega')$, so $\cC\subset \cC'$. At $(X', \omega')$, the relative cohomology class $u$ still describes a twist in the cylinders of $\cC'$. This twist involves only the cylinders in $\cC$, and not those in $\cC'\setminus \cC$. 

Let $u'$ be the standard twist at $(X', \omega')$ in the cylinders of $\cC'$. By Theorem \ref{T:twists}, by rescaling $u'$ if necessary, we may assume $p(u)=p(u')$. 

Suppose the core curves of the cylinders in $\cC$ are $\alpha_i$, and the core curves of the cylinders of $\cC'\setminus \cC$ are $\beta_i$. Assume for convinience that the cylinders under consideration are all horizontal.  We have
$$u=\sum h_i I_\Sigma(\alpha_i) \quad\text{and}\quad u'=\sum h_i' I_\Sigma(\alpha_i) + \sum \delta_i I_\Sigma(\beta_i),$$
where the $\delta_i$ are the heights of the new cylinders, and the $h_i$ and $h_i'$ are the heights of the cylinders in $\cC$ at $(X, \omega)$ and $(X', \omega')$ respectively.  Since these are equal in absolute homology, we get 
$$\sum h_i \alpha_i = \sum h_i' \alpha_i + \sum \delta_j \beta_j$$
in absolute homology. Lemma \ref{L:hom} now gives the result. 
\end{proof}

\begin{proof}[\eb{Proof of Proposition \ref{P:Cbound}.}]
Suppose to the contrary that there is a sequence of surfaces $(X_n, \omega_n)\in \cM^1$ that has an equivalence class $\cC_n$ of $\cM$-parallel cylinders where the ratio of the largest to the smallest circumference goes to infinity. Without loss of generality, assume that all these $\cC_n$ are horizontal, and that there are no horizontal saddle connections on $(X_n, \omega_n)$ that are not $\cM$-parallel to the cylinders in $\cC_n$. (If there are horizontal saddle connections on $(X_n, \omega_n)$ that are not $\cM$-parallel to the cylinders in $\cC_n$, the surface can be deformed slightly within $\cM$ to remove this unnecessary coincidence.)  

Let $w_n$ be the shortest horizontal saddle connection on $(X_n, \omega_n)$. By Proposition \ref{P:notinySC}, there is a constant $R_{sc}$ such that there is a cylinder in $\cC_n$ that has circumference at most $R_{sc}$ times the length of $w_n$, and that has area at least $1/R_{sc}$ times the total area of $\cC_n$. As in the proof of Proposition \ref{P:notinySC}, we can assume that this cylinder has circumference 1 and area at least $1/(2R_{sc})$. Since the  smallest horizontal saddle connection on $(X_n, \omega_n)$ has length at least $1/R_{sc}$, using Theorem \ref{T:MW} we may assume that $(X_n, \omega_n)$ converges to some surface $(X, \omega)$ in the same stratum. 

There is at least one cylinder  $C$ on $(X, \omega)$ that is a limit of horizontal cylinders on $(X_n,\omega_n)$ of circumference 1 and area at least $1/(2R_{sc})$. Let $\cC$ be the equivalence class of $C$. 

Let $Q$ be the sum of the circumferences of the cylinders in $\cC$. We now claim that for large $n$, all cylinders in $\cC_n$ are size at most $Q+1$, which contradicts our original supposition that the ratio of the largest and smallest circumferences of cylinders in $\cC_n$ goes to infinity. 

The claim follows directly from Lemma \ref{L:atmostsum}.
\end{proof}

\subsection{Proof the Cylinder Finiteness Theorem.}

\begin{lem}\label{L:homfin}
For any simply connected compact set $K$ of a stratum, and any constant $L>0$,  the set of relative homology classes represented by a saddle connection of length at most $L$ on some $(X,\omega)\in K$  is finite. 
\end{lem}

The fact that $K$ is simply connected means that the relative homology of different surfaces in $X$ can be canonically identified. The proof is standard, but we give a sketch. 

\begin{proof}
There is some $c>1$ depending on $K$ so that if a relative cohomology class is represented by a saddle connection of length $L$ somewhere in $K$, then it is represented by a union of saddle connections of total length at most $cL$ everywhere on $K$.

If there were infinitely many relative cohomology classes represented somewhere on $K$ by saddle connections of length at most $L$, then on every $(X,\omega)\in K$ all of these relative cohomology classes would be represented by unions of saddle connections of total length  at most $cL$. This contradicts the fact that there are at most finitely many saddle connections on $(X,\omega)$ of length at most $cL$. 
\end{proof}

\begin{lem}\label{L:cV}
Let $\cU$ be a set of translation surfaces with compact closure in the stratum. Let $L>0$ be a large constant. Then $\cU$ is contained in a finite union of {connected} subsets $\cV$, such that  the following properties hold.
\begin{itemize}
\item Suppose $(X, \omega)\in \cV$ has a saddle connection of length at most $L$. Then this saddle connection exists everywhere in $\cV$. 
\item Suppose $(X, \omega)\in \cV$ has a pair saddle connection of length at most $L$ that are parallel. Then these saddle connections remain parallel on all of $\cV$. 
\end{itemize}
\end{lem}

\begin{proof}
It suffices to show this for any small set $\cU$ that is a round ball in some choice of local period coordinates. 

Let $P\subset H_1(X, \Sigma, \bZ)$ be the set of relative homology classes that are represented by saddle connections of length  at most $cL$ at some surface in $\cU$, where $c$ is the constant from the proof of Lemma \ref{L:homfin}.  This set is finite by Lemma \ref{L:homfin}. For each $\alpha, \beta\in P$, consider the set $H_{\alpha, \beta}\subset \cU$ where the holonomy of $\alpha$ and $\beta$ are parallel. In local coordinates, this set is a hypersurface defined by a single quadratic equation. 

For each subset of pairs $J \subset P\times P$, consider the subset 
$$\cV_J=\bigcap_{(\alpha, \beta)\in J} H_{\alpha, \beta} \setminus \bigcup_{(\alpha, \beta)\notin J} H_{\alpha, \beta}$$
 of surface in $\cU$ that are in $H_{\alpha, \beta}$ if and only if $(\alpha, \beta)\in J$. Each $\cV_J$ intersects $\cU$ in at most finitely many connected components. (Indeed, these intersections are semi-algebraic sets, which always have at most finitely many connected components  \cite[Theorem 2.4.5]{BCR}.)

It is clear that the connected components of  $\cV_J$ satisfy the two desired properties, since when a saddle connection is created or destroyed a pair of previously parallel saddle connections become not parallel or vice versa.  
\end{proof}

\begin{cor}
Let $\cM$ be an affine invariant submanifold and let $L'>0$. Then each compact subset $K \subset \cM$ can be covered by finitely many connected sets $\cV$, such that the following property holds.  If $(X, \omega)\in \cV$ and $C$ is a cylinder of length at most $L'$ on $(X, \omega)$, then the set of all cylinders $\cM$-parallel to $C$ is constant on $\cV$, and the set of boundary saddle connections of these cylinders is also constant on $\cV$.
\end{cor}

\begin{proof}
This follows from Lemma \ref{L:cV} with constant $L$ equal to $L'$ times any upper bound $R_{cyl}$ for the ratio  of two $\cM$-parallel cylinders in $\cM$, which exists by Proposition \ref{P:Cbound}.
\end{proof}

Finally, we can conclude the proof of Theorem \ref{T:CFT}. 

\begin{proof}[\eb{Proof of Theorem \ref{T:CFT}.}]
Pick $\e=1/R_{sc}$, where $R_{sc}$ is the constant from Proposition \ref{P:notinySC}. Let $K\subset \cH^1$ be the compact set given by Theorem \ref{T:MW} such that every horocycle orbit of a unit area surface without saddle connections of length less than $\e$ intersects $K$.

Cover $K$ by finitely many connected sets $\cV$ as in the previous corollary, using $L'=1$.  For each $\cV$ and each cylinder $C$ of circumference at most $1$ on a surface in $\cV$, this cylinder and the equivalence class $\cC$ of cylinders $\cM$-parallel to $C$ persist everywhere on $\cV$, and the ratio of circumferences of cylinders\ann{A: Typo fixed.} in $\cC$ is constant on $\cV$. There are only finitely many such $C$ and $\cC$ for each $\cV$. 
 
 
 Suppose that the cylinders in $\cC$ have core curves $\alpha_i$ and circumferences $c_i$. Each twist in this equivalence class can be written as   
$$\sum t_i c_i I_\Sigma(\alpha_i)$$
for some $t_i$ in $\bC$. The $t_i$ are coordinates for the set of all twists supported on $\cC$. Note the circumferences $c_i$ are not constant on $\cV$, but their ratios are. The twist space of $\cC$ is also constant on $\cV$. 

Define the sets $S_1$ and $S_2$ as follows. For every $\cV$, and every equivalence class $\cC$ of $\cM$-parallel cylinders that contains a cylinder of circumference at most 1 somewhere on $\cV$, add the ratios of the circumferences to $S_1$. By Lemma \ref{L:rational}, the twist space of $\cC$ can be defined by linear equations on the $t_i$ with coefficients in $\bQ$. Pick a finite set of such equations defining the twists space of $\cC$, and add all the coefficients to $S_2$. (The definition of $S_2$ is not canonical.)

We will prove the theorem with this choice of $S_1$ and $S_2$. Take any surface $(X, \omega) \in \cM^1$, and any equivalence class $\cC$ of $\cM$-parallel cylinders on $(X, \omega)$. Without loss of generality, assume these cylinders are horizontal. 

If there are horizontal saddle connections not $\cM$-parallel to the cylinders in $\cC$, nudge $(X, \omega)$ to get $(X', \omega')$ where all cylinders in $\cC$ persist and stay horizontal, but where there are no such horizontal saddle connections accidentally parallel to the cylinders in $\cC$. Let $\cC'$ be the set of cylinders at $(X', \omega')$ that are $\cM$-parallel to cylinders in $\cC$, so $\cC\subset \cC'$. Note that the twist space for $\cC$ at $(X, \omega)$ embeds naturally in the twist space for $\cC'$ at $(X', \omega')$ via parallel transport. Furthermore, in the $t_i$ coordinates, the twist space of $\cC$ is a coordinate subspace of the twist space for $\cC'$, because twists in $\cC'$ are parallel transports of twists in $\cC$ if and only if they do not involve the cylinders of $\cC'\setminus \cC$, i.e. if the $t_i$ corresponding to cylinders in $\cC'\setminus \cC$ are $0$. 

By Proposition \ref{P:notinySC}, if we pick $t$ such that the smallest circumference of a cylinder in $\cC'$ on $g_t(X', \omega')$ is exactly 1, then  the shortest horizontal saddle connection is at most $1/R_{sc}$.  By Theorem \ref{T:MW},  we can find $s$ so that $h_s g_t(X', \omega')\in K$. By the definition of $S_1$ and $S_2$, this gives that the circumferences and twist space of $\cC'$ and hence also $\cC$ are described by the finite sets $S_1$ and $S_2$ as desired. 
\end{proof}

\section{Recognizable twists}\label{S:Rec}

The goal of this section is to prove that certain \emph{recognizable} twists span $T(\cM)$, and that a recognizable twist at a surface is still recognizable at nearby surfaces.  This ability to robustly recognize the tangent space will allow us to compare the tangent space at a surface in $\cM$ to the tangent space at a surface in $\partial \cM$ in the next section. 

\begin{rem}\label{R:E}
Before discussing what does work, it is helpful to mention what doesn't work. We could simply consider, for each surface $(X,\omega)$, the span $E(X,\omega)$ of the standard twists in the full set of all cylinders in each direction. (Or, what is a bit better, we could also include twists that are in the closure of the one parameter families of deformations given by the standard twists. So, for example if a direction on $(X,\omega)$ has exactly two parallel cylinders, but their moduli aren't rational multiples of each other, this would include the twist in each of these cylinders separately.) 

This $E(X,\omega)$ is defined without reference to any particular $\cM$ and is very poorly behaved. For example, square-tiled surfaces are dense in every stratum, and for each square-tiled surface $E(X,\omega)$ is just the tautological plane $\span(\Re(\omega), \Im(\omega))$. In particular, $E(X, \omega)$ is nowhere continuous. Even worse, as far as we know it is possible for $E(X, \omega)$ to be very small even if the orbit closure of $(X,\omega)$ is large. For example, it is an open question whether $E(X, \omega)=\span(\Re(\omega), \Im(\omega))$ implies that $(X,\omega)$ lies on a closed orbit \cite[Question 5]{SWprobs}. (For some progress on this question, see \cite{LNW}.)
\end{rem}

\subsection{Definitions and examples.}

We will begin in a fairly abstract setting, without reference to $\cM$.\ann{R: ¤6.1 is too abstract. It took me hours to understand what it means.\\A: Lemma \ref{L:inM} moved earlier in text, and referenced here, to make it less abstract.} The reader may  consult Lemma \ref{L:inM} as they are reading the definitions for some motivation. 

Let $S_1, S_2$ be two finite sets of real numbers. Suppose that $(X,\omega)$ is a translation surface, and let $\{C_i\}$ be the  set of {all} cylinders in some direction on $(X, \omega)$. Suppose $C_i$ has height $h_i$, circumference $c_i$,  modulus $m_i=h_i/c_i$, and core curve $\alpha_i$.
   
   We will say that the twist $$u=\sum t_i c_i I_\Sigma(\alpha_i)$$ is $(S_1, S_2)$-recognizable if, for all subsets of the cylinders $C_i$ with all ratios of circumferences in $S_1$,  the corresponding $t_i$ satisfy all linear equations with coefficients in $S_2$ that the $m_i$ do. 
   
More precisely, we require that for any subset  $\{C_i: i \in \cI\}$ of the cylinders with $c_i/c_j\in S_1$ for all $i,j\in \cI$, and any choice of $f_i\in S_2$ for each $i \in \cI$,  $$\sum_\cI f_i m_i=0\quad\quad \text{implies}\quad\quad   \sum_\cI f_i t_i=0.$$
   
Note that the $c_i=c_i(X,\omega)$ and $m_i=m_i(X,\omega)$ vary with $(X,\omega)$, and so a twist can be $(S_1, S_2)$-recognizable at a surface but not at some nearby surface. 
   
\begin{ex}
Suppose that in some direction $(X,\omega)$ has three cylinders, with circumferences $c_1=1, c_2=1, c_3=40$ and moduli $m_1=1, m_2=2, m_3=1$. Suppose $S_1=S_2=\{1\}$. Then the twist described by $t_1=1, t_2=-3, t_3=0$ is $(S_1, S_2)$-recognizable. To check this, note that the only set of cylinders with all ratios of circumferences in $S_1$ is $\{C_1, C_2\}$. So we need to check that $t_1, t_2$ satisfy all linear equations with coefficients in $S_2$ that $m_1, m_2$ satisfy. In this case $m_1, m_2$ don't satisfy any linear equations with coefficients $S_2$.  

 Next consider the same example but with $S_1=\{1,40, 1/40\}$ and $S_2=\{1\}$. The full set of all three cylinders has all ratios of circumferences in $S_1$. So we need to check that $t_1, t_2, t_3$ satisfy all linear equations with coefficients in $S_2$ that the $m_1, m_2, m_3$ satisfy. In this case the equation $m_1=m_3$ holds, but $t_1\neq t_3$, so the twist is not recognizable.  
  
Next consider the same example but with $S_1=\{1\}, S_2=\{1,2\}$.  The only set of cylinders with all ratios of circumferences in $S_1$ is $\{C_1, C_2\}$, but now the equation $m_2=2m_1$ shows the twist is not $(S_1,S_2)$-recognizable. 
\end{ex}

\begin{ex} If $\cC$ is the collection of all cylinders in some direction on $(X, \omega)$, then the standard twist in $\cC$ is $(S_1, S_2)$-recognizable for any  $S_1, S_2$. This is because this twist corresponds to $t_i=m_i$. 
\end{ex}

\begin{rem}
For many surfaces the standard twists will be the only $(S_1, S_2)$-recognizable twists. However at special surfaces, for example when there are extra parallel cylinders in some direction, there may be $(S_1, S_2)$-recognizable twists that are not standard. 
\end{rem}

\subsection{Key results.}

\begin{lem}\label{L:inM}
For any affine invariant submanifold $\cM$, there are finite sets $S_1, S_2$ so that any $(S_1, S_2)$-recognizable twist at a surface in $\cM$ remains in $\cM$. 
\end{lem}

Pairs of sets $(S_1, S_2)$ satisfying the conclusion of the lemma will be called $\cM$-adapted. Note that enlarging an adapted set always gives another adapted set. 

\begin{proof}
Take the $S_1, S_2$ given by the Cylinder Finiteness Theorem. 

Suppose that $u$ is a $(S_1, S_2)$-recognizable twist. It suffices to assume it is supported on a equivalence class of $\cM$-parallel cylinders, since $u$ is the sum of its restrictions to such equivalence classes. 

Write $u=\sum t_i c_i I_\Sigma(\alpha_i)$. By assumption, the twist space of the equivalence class of $\cM$-parallel cylinders can be defined by linear equations in the $t_i$  with coefficients in $S_2$. The twist space contains the standard twist given by $t_i=m_i$, so hence all of the equations defining the twist space must have $t_i=m_i$ as a solution. 

By the definition of $(S_1, S_2)$-recognizable, the $t_i$ defining $u$ satisfy all such linear equations. 
\end{proof}

The key technical property of recognizable twists is given by the next lemma.  

\begin{lem}\label{L:recnearby}
Suppose that $u$ is a $(S_1, S_2)$-recognizable twist at $(X, \omega)$. Then there is a neighborhood $\cV$ of $(X, \omega)$ such that for all $(X', \omega')\in \cV$, the parallel translate of $u$ to $(X', \omega')$ is in the  span of the $(S_1, S_2)$-recognizable twists at $(X', \omega')$. 
\end{lem}

By span we always mean span over the complex numbers. 

\begin{proof}
Let $C_i, i=1, \ldots, k$ be the collection of all cylinders in the direction of the cylinder twist $u$. Say $$u=\sum_{i=1}^k t_ic_i(X,\omega) I_\Sigma(\alpha_i).$$

Let $n$ be the maximum number of parallel cylinders on a surface in the stratum of $(X, \omega)$. Consider $$(m_1(X,\omega), \ldots, m_k(X,\omega), 0, \ldots, 0) \in \bR^{k+n}.$$ Find a neighborhood $\cU\subset \bR^{k+n}$ of this point such that for all $(m_1', \ldots, m_{k+n}')\in \cU$, the $m_i'$ do not satisfy any linear equations with coefficients in $S_2$ except those of the form $E_1+E_2=0$, where $E_1$ is a linear equation with coefficients in $S_2$ satisfied by  $(m_1(X,\omega), \ldots, m_k(X,\omega))$, and $E_2$ is any linear equation in the remaining $n$ variables with coefficients in $S_2$. 


We may find a neighborhood $\cU_0\times B_\e \subset \cU$, where $\cU_0$ is a neighborhood of $(m_1(X,\omega), \ldots, m_k(X,\omega))$ in $\bR^k$, and $B_\e$ is the ball of radius $\e$ in the remaining coordinates $\bR^{n}$. 

Define $\cV$ to be the set of  $(X', \omega')$ sufficiently close to $(X, \omega)$ so that the following three conditions hold.
\begin{enumerate} 
\item All the $C_i$ persist at $(X', \omega')$ and  the moduli $m_i(X',\omega')$ of the $C_i$ at $(X', \omega')$ have $(m_1(X',\omega'), \ldots, m_k(X',\omega'))\in \cU_0$.
\item If $C$ is a cylinder parallel to $C_i$ at $(X', \omega')$, and $C$ is not one of the $C_j$, and there is some $j$ so that the ratio of the circumference of $C$ and $C_j$ is in $S_1$, then the modulus of $C$ is less than $\epsilon$.
 \item If $C_i$ and $C_j$ have ratio of circumferences in $S_1$ at $(X', \omega')$, then they have the same ratio of circumferences at $(X, \omega)$. 
\end{enumerate}

Consider the  equivalence relation on the $C_i$, generated by $C_i\sim C_j$ if $C_i$ and $C_j$ remain parallel on $(X', \omega')$ and the ratio of their circumferences at $(X', \omega')$ is in $S_1$. Let $C_i, i\in \cI$ be an equivalence class for this equivalence relation. 

\bold{Claim:}  For each such equivalence class $\cI$,  the twist
$$u_{\cI}=\sum_{i\in \cI} t_i c_i(X, \omega) I_\Sigma(\alpha_{i})$$ is a recognizable twist at $(X', \omega')$. 

\bold{Conclusion of the proof of lemma given the claim:} 
$u$ is the sum of the twists $u_{\cI}$ over the different equivalence classes $\cI$. Since each $u_{\cI}$ is $(S_1, S_2)$-recognizable at $(X', \omega')$, we see that $u$ is in the span of the recognizable twists at $(X', \omega')$.

\bold{Proof of claim:} Because the ratios of circumferences of the $C_i, i\in \cI$ have not changed in moving from $(X, \omega)$ to $(X', \omega')$, the cohomology class $u_{\cI}$ is proportional to 
$$\sum_{i\in \cI} t_i c_i(X', \omega') I_\Sigma(\alpha_{i}).$$
Suppose, in order to find a contradiction, that this twist  is not recognizable at $(X', \omega')$. By definition, there must be 
\begin{itemize}
\item a collection of cylinders on $(X', \omega')$ parallel to the $C_i, i \in \cI$ for which all ratios of circumferences are in $S_1$, and  
\item  a linear equation $E$ with coefficients in $S_2$  that is satisfied by the moduli of these cylinders but not the $t_i$.
\end{itemize}
The $t_i$ value for any cylinder  other than the $C_i, i \in \cI$ is 0, because only the cylinders $C_i, i \in \cI$ appear in the definition of $u_{\cI}$.
\ann{R: I do not understand the last paragraph of the proof of Lemma 6.5. The lemma now follows directly from the claim ...\\A: The proof has been restructured and expanded to clarify. (What was Lemma 6.5 is now Lemma 6.6.)}

Let $D_j, j\in \cJ$ be the set of cylinders on $(X', \omega')$ that are parallel to some $C_i, i\in \cI$ and have ratio of circumference with this $C_i$ in $S_1$, but which are not equal to any of the $C_i$ (the full set of parallel cylinders considered on $(X, \omega)$). By part (2) of the definition of $\cV$, all $D_j$ have  modulus less than $\e$. By part (3),  the set of $m_i(X', \omega'), i \in \cI$ and $m_j(X', \omega'), j \in \cJ$ do not  satisfy any linear equations with coefficients in $S_2$, except those of the form $E_1+E_2$, where $E_1$ is an equation on the $t_i, i\in \cI$ that holds for the $m_i(X, \omega)$, and $E_2$ is any equation with coefficients in $S_2$ on the new $t$-variables corresponding to the $D_j, j\in \cJ$. So we may assume $E=E_1+E_2$ has this form.

The $t_i, i\in \cI$ satisfy  $E_1$ by construction, and  $u_{\cI}$ has 0 value for the new $t$-variables corresponding to the $D_j, j\in \cJ$. Hence $E$ is satisfied by the $t_i$, and we get a contradiction. We conclude that $u_{\cI}$ is in fact recognizable at $(X', \omega')$. 
 \end{proof}

%
%
%
%
%

\begin{lem}\label{L:adapted}
Let $(S_1, S_2)$ be arbitrary. There is an open, $GL(2,\bR)$ invariant set in $\cM$ on which the span of the $(S_1, S_2)$-recognizable twists is locally constant. 
\end{lem}

\begin{proof}
Consider the open, $GL(2,\bR)$ invariant set where the span of the $(S_1, S_2)$-recognizable twists has maximal dimension. On this locus, the span is locally constant by Lemma \ref{L:recnearby}. 
\end{proof}

\begin{prop}\label{P:spanTM}
Suppose $(S_1, S_2)$ is $\cM$-adapted. There is a finite union of proper affine invariant submanifolds of $\cM$ such that if $(X, \omega)\in \cM$ is not contained in this finite union then the $(S_1,S_2)$-recognizable twists at $(X,\omega)$ generate the tangent space to $\cM$ at $(X,\omega)$. 
\end{prop}

\begin{proof}
Consider the open $GL(2,\bR)$ invariant set of Lemma \ref{L:adapted}. By EMM, the complement of this set is a finite union of proper affine invariant submanifolds. 


The span of the $(S_1, S_2)$-recognizable twists gives a flat subbundle over $\cM$ minus this finite union of proper orbit closures, and this flat subbundle is not contained in $\ker(p)\cap T(\cM)$, because standard twists are never in $\ker(p)$. 

By \cite[Theorem 5.1]{Wfield}, any such flat subbundle must be all of $T(\cM)$.  
\end{proof}

\subsection{$(S_1,S_2)$ almost determines $\cM$.}

The following result is not used in this paper. It is included to develop intuition, and because it may have applications in the future. 

\begin{thm}\label{T:Sdetermines}
Fix finite sets $S_1, S_2\subset \bR$. In any fixed stratum there are at most finitely many affine invariant submanifolds $\cM$ such that $(S_1, S_2)$ is $\cM$-adapted. 
\end{thm}

\begin{proof}
Otherwise, by EMM we would get an infinite sequence of $\cM_n$ for which $(S_1, S_2)$ is $\cM$-adapted and that are dense in some larger $\cM$. By Lemma \ref{L:adapted}, there is a finite union of affine invariant submanifolds $\cN_i$ in $\cM$, such that on $\cM\setminus \cup \cN_i$ the span of the $(S_1, S_2)$-recognizable twists is locally constant. 

At most points of $\cM_n$, the $(S_1, S_2)$-recognizable twists span $T(\cM_n)$. This gives a contradiction, since the $\cM_n$ are dense in $\cM$, and two different affine invariant submanifolds cannot have the same tangent space. (More precisely, given two nearby points $(X, \omega)\in \cM_n$ and $(X', \omega')\in \cM_{n'}$ the parallel translate from $(X, \omega)$ to $(X', \omega')$ of $T(\cM_n)$ cannot be equal to $T(\cM_{n'}$).)
\end{proof}

\section{Boundary (connected case)}\label{S:Boundary}

In the first subsection we show that each point in $\partial \cM \cap \overline{\cH}_{conn}$ is contained in an affine invariant submanifold with the desired tangent space. In the second subsection, we bootstrap this to give a global statement about $\partial \cM \cap \overline{\cH}_{conn}$. 

\subsection{Affine invariant submanifolds containing limits.} 

Although the rest of this section deals only with connected degenerations, the following proposition applies to all of $\overline{\cH}$, not just $\overline{\cH}_{conn}$.

\begin{prop}\label{P:rinside}
Suppose that $u$ is a $(S_1, S_2)$-recognizable twist at $(X, \omega)\in\overline{\cH}$, and that $(X_n, \omega_n)\in \cH$ is a sequence  converging to $(X, \omega)$. Then, for $n$ large enough, $u$ is in the span of the $(S_1, S_2)$-recognizable twists at $(X_n, \omega_n)$. 
\end{prop}

Let $V\subset H_1(X,\Sigma, \bZ)$ denote the space of vanishing cycles. Recall that the tangent space to the boundary stratum is naturally identified with $\Ann(V)\subset H^1(X,\Sigma, \bC)$. It is this identification that we are using to consider $u$ as a relative homology class at $(X_n, \omega_n)$. In particular, note that at $(X_n, \omega_n)$, by definition  $u$ is in $\Ann(V)$.

\begin{proof}
This follows  exactly as in the proof of Lemma \ref{L:recnearby}, using Lemma \ref{L:horrible}.\ann{A: Added reference to Lemma \ref{L:horrible}.}
\end{proof}

\begin{prop}\label{P:main}
Suppose that $(X_n, \omega_n)\in \cM$ converge to $(X, \omega)\in \partial \cM \cap \overline{\cH}_{conn}$. 

Then $(X, \omega)$ is contained in an affine invariant submanifold whose tangent space at $(X,\omega)$ is given, for infinitely many $n$, by $\Ann(V)\cap T_{(X_n, \omega_n)}(\cM)\subset H^1(X_n,\Sigma_n, \bC)$ via the isomorphism between the tangent space to the boundary stratum at $(X,\omega)$ and $\Ann(V)\subset H^1(X_n,\Sigma_n, \bC)$. 
\end{prop}

 \begin{rem}
 It could be that $(X,\omega)$ is also contained in a smaller affine invariant submanifold even if all $(X_n, \omega_n)$ have orbit closure $\cM$, although this is not the generic situation. 
 \end{rem}
 
 \begin{proof}
By Proposition \ref{P:AnnV}, there is a neighborhood of $0$ in $\Ann(V)$ such that for $\xi$ in this neighborhood and $n$ large enough, $(X_n, \omega_n)+\xi$ and $(X,\omega)+\xi$ are well defined, and such that if $\xi_n\in \Ann(V)$ are in this neighborhood and $\xi_n \to \xi$, then $(X_n, \omega_n)+\xi_n$ converges to $(X,\omega)+\xi$. 

By Proposition \ref{P:Vn}, each subspace $T_{(X_n, \omega_n)}(\cM)\cap \Ann(V)\subset H^1(X_n,\Sigma_n, \bC)$ can be thought of as contained in a fixed vector space isomorphic to the tangent space of the boundary stratum. It suffices to pass to a subsequence along which these subspaces converge in this fixed vector space. Assume we have already passed to this subsequence. 

Let $U$ denote the set of all limits of all such sequences $(X_n, \omega_n)+\xi_n$, where $\xi_n \in \Ann(V) \cap T_{(X_n, \omega_n)}(\cM)$ is in the neighborhood of $0$ referred to above. Note that $U$ is equal to $(X,\omega)$ plus all sufficiently small vectors in the subspace of $\Ann(V)$ to which $T_{(X_n, \omega_n)}(\cM)\cap \Ann(V)\subset H^1(X_n,\Sigma_n, \bC)$ converges.  

Let $\cM'$ be the smallest closed $GL(2,\bR)$ invariant set containing $U$. Because all closed $GL(2,\bR)$ invariant sets are finite unions of affine invariant submanifolds, and because $U$ is linear, we see that $\cM'$ is a single affine invariant submanifold. 

Let $(S_1, S_2)$ be both $\cM$-adapted and $\cM'$-adapted. By Proposition \ref{P:spanTM}, away from a finite union of proper affine invariant submanifolds, the tangent space to $\cM'$ is spanned by $(S_1, S_2)$-recognizable twists. 

Pick a specific sequence $\xi_n$ so that the limit of the $(X_n, \omega_n)+\xi_n$ is a surface $(X', \omega')$ not contained in this finite union of proper affine invariant submanifolds. 
 
 The tangent space to $\cM'$ at $(X', \omega')$ is spanned by $(S_1, S_2)$-recognizable twists.  By Proposition \ref{P:rinside}, all the  $(S_1, S_2)$-recognizable twists at $(X', \omega')$ are recognizable at $(X_n, \omega_n)+\xi_n$ as well. Hence they are all in $\Ann(V)\cap T_{(X_n,\omega_n)}(\cM)$. 

We have shown that the tangent space to $\cM'$ at $(X', \omega')$ is contained in $\Ann(V)\cap T_{(X_n,\omega_n)}(\cM)$ for all $n$ large enough. Let us now address the opposite containment. Suppose $T(\cM')$ was strictly contained in $\Ann(V)\cap T_{(X_n,\omega_n)}(\cM)$ for all $n$.  In this case, pick $\xi_n'\in \Ann(V)\cap T_{(X_n,\omega_n)}(\cM)$ so that $\xi_n'$ converges to $\xi \notin T(\cM')$. The definition of $\cM'$ gives that for $t$ small  enough, $(X', \omega')+t\xi\in \cM$. Hence $\xi\in T(\cM')$, a contradiction. 
\end{proof}

\begin{rem}\label{R:variant}
A variant of the above approach allows the use of sets $(S_1, S_2)$ that are $\cM$-adapted without knowing whether they are $\cM'$-adapted. The proof of Proposition \ref{P:spanTM} shows that for any affine invariant submanifold $\cM'$, and  arbitrary $(S_1, S_2)$, away from a finite union of proper affine invariant submanifolds the span of the $(S_1, S_2)$-recognizable twists contains $T(\cM')$. In the above proof, one can then pick $(X', \omega')$ such that the span of the  $(S_1, S_2)$-recognizable twists contains $T(\cM')$. This strategy will be required in Section \ref{S:MultiCase}.
\end{rem}

\subsection{Global boundary}

We now rule out the possibility that infinitely many boundary orbit closures $\cM'$ occur in Proposition \ref{P:main}. This will conclude the proof of the first part of Theorem \ref{T:main2}, which appeared in the introduction as Theorem \ref{T:main}.

\begin{thm}\label{T:finite}
Only finitely many affine invariant submanifolds arise in Proposition \ref{P:main}. 
\end{thm}\ann{R: The statement of Theorem 7.5 is not clear.\\ A: Added clarification immediately after the theorem statement.}

That is, there is a finite list of affine invariant submanifolds in the boundary of $\cM$ such that if $(X_n, \omega_n)\in \cM$ converge to $(X, \omega)\in \partial \cM \cap \overline{\cH}_{conn}$, then  $(X, \omega)$ is contained in one of these finitely many affine invariant submanifolds, and the tangent space to this affine invariant submanifold  can be computed as in Proposition \ref{P:main}.

\begin{proof}
Suppose in order to find a contradiction that infinitely many affine invariant submanifolds  arise in Proposition \ref{P:main}. Pick a sequence $\cM_k$ of them that all have maximal dimension and all lie in the same boundary stratum. 

By EMM, there are $(X^{(k)}, \omega^{(k)})\in \cM'_k$ that converge to some $(X^{(\infty)},\omega^{(\infty)})$ that has  larger dimensional orbit closure $\cM'_\infty$. 

Pick a sequence $(X_n^{(k)}, \omega_n^{(k)})\to (X^{(k)}, \omega^{(k)})$ with $(X_n^{(k)}, \omega_n^{(k)})\in \cM$ and such that $T(\cM'_k)=\Ann(V)\cap T_{(X_n^{(k)}, \omega_n^{(k)})}(\cM)$.

 Some ``diagonal" sequence $(X_{n_k}^{(k)}, \omega_{n_k}^{(k)})$ then converges to $(X^{(\infty)},\omega^{(\infty)})$. Thus, for infinitely many $n$, the tangent space to both $\cM'_k$ and $\cM'_\infty$ can be identified with $$\Ann(V)\cap T_{(X_{n_k}^{(k)}, \omega_{n_k}^{(k)})} \cM,$$ which is a contradiction because $\cM'_k$ and $\cM'_\infty$ have different dimensions.  

 Proposition \ref{P:Vn} allows us to consider $V$ without worrying about dependence on $k$ or $n_k$. 
\end{proof}

\begin{proof}[Proof of Theorem \ref{T:main2}, first part] 
Proposition \ref{P:main} shows that each point in the boundary appears in an affine invariant submanifold with the desired tangent space. The proof additionally shows that the set of limit points are dense in the affine invariant submanifold. 

Theorem \ref{T:finite} shows that only finitely many affine invariant submanifolds appear in this way. Hence these finitely many affine invariant submanifolds constitute the boundary of $\cM$ in the given boundary stratum. 
\end{proof}

\section{Multicomponent case}\label{S:MultiCase}

In the first subsection, we prove the second statement in Theorem \ref{T:main2}, which is the direct generalization of Proposition \ref{P:main} to multicomponent surfaces, by explaining what modifications to the proof of Proposition \ref{P:main} are required. In the next two subsections, we assume  multicomponent EMM is true, and prove  Theorem \ref{T:main3}. As in the connected case, this involves first showing that all limits are contained in an affine invariant submanifold with the expected tangent space (as in Proposition \ref{P:main}) and then showing that only finitely many such affine invariant submanifolds occur (as in Theorem \ref{T:finite}).

\subsection{Projection onto connected components}

Consider a multicomponent limit  $(X', \omega')$ of a sequence $(X_n, \omega_n)\in \cM$.  Say $\pi$ is the projection onto a connected component. (So $\pi$ just forgets the other connected components.) 

As in the proof of Proposition \ref{P:main}, after passing to an appropriate subsequence, we consider the smallest affine invariant submanifold $\cM'$ that contains all $\pi(X', \omega')$, where $(X', \omega')$ is a limit of $(X_n, \omega_n)+\xi_n$, as in the proof of Proposition \ref{P:main}. 

As in that proposition, we can assume $(X', \omega')$ has the property that $T(\cM')$ is spanned by $(S_1, S_2)$-recognizable twists at $\pi(X', \omega')$. We now claim that for every $(S_1, S_2)$-recognizable twist $u$ on $\pi(X', \omega')$, there is a $(S_1, S_2)$-recognizable twist on $(X', \omega')$ which restricts to $u$.

This is a consequence of simple linear algebra. Let $\cE$ be a set of homogeneous linear equations in variables $t_1, \ldots, t_n$. Suppose that $t_1, \ldots, t_k$ satisfy all the linear equations in $\cE$ that only involve the first $k$ variables. Then there are values of $t_{k+1}, \ldots, t_n$ so that $t_1, \ldots, t_n$ satisfy all linear equations in $\cE$. 

Thus, also as in Proposition \ref{P:main}, we see that $T(\cM')$ is contained in $\pi(\Ann(V)\cap T(\cM))$. The proof that $T(\cM')$ cannot be smaller than $\pi(\Ann(V)\cap T(\cM))$ is the same as in the connected case (Proposition \ref{P:main}).

\subsection{Results using multicomponent EMM: Analogue of Proposition \ref{P:main}}

Note that every affine invariant submanifold is a product of irreducible affine invariant submanifolds. 

We follow the argument in the connected case, however some complications arise. As in the proof of Proposition \ref{P:main}, we consider the set $U$ of limits of sequences $(X_n, \omega_n)+\xi_n$, where $\xi_n\in \Ann(V)\cap T_{(X_n,\omega_n)}(\cM)$ are small enough. We may assume that $U$ is a small locally linear set in the boundary. The main complication is the possibility, which  is ruled out only at the conclusion of our arguments, that the smallest closed $GL(2, \bR)$ invariant set containing $U$ is not an affine invariant submanifold. 

 To be concrete, 
  let us mention that one such bad situation would be if the smallest closed invariant set containing $U$ was the set of all surfaces in $\cH(2)\times \cH(2)$ where the two components have equal area.  The following is an example of a locally linear set $U$ with exactly this bad property. 
   
   \ann{R: I do not understand Example 8.1.\\A: Added explanation.}
 \begin{ex}\label{E:H2}
  Choose $(X,\omega)\in \cH(2)$ with dense orbit, and consider $v\in H^1(X,\bR)$ that is symplectically orthogonal to $\Re(\omega)$ and $\Im(\omega)$. Note that for any small such $v$, $(X, \omega)+v$ and $(X, \omega)$ have the same area. 
 
Define $U$ to be the subset of $\cH(2)\times \cH(2)$ given by $((X,\omega), (X,\omega))+\xi$, where $\xi$ ranges over a neighborhood of zero in $$\span_\bC((\Re(\omega), \Re(\omega)), (\Im(\omega), \Im(\omega)), (0, v)).$$ Note that every element of $U$ has that the two components have equal area. Assuming multicomponent EMM, one can show that the smallest $GL(2, \bR)$ invariant set containing $U$ is  the set of all surfaces in $\cH(2)\times \cH(2)$ where the two components have equal area. (Indeed, by Conjecture \ref{C:MEMM}, the smallest closed invariant set containing $U$ must be the subset of some affine invariant submanifold where the two components have equal area. By construction and part (2) of the conjecture, this affine invariant submanifold can't be irreducible. Since by construction its projection to each coordinate must be all of $\cH(2)$, we see that it must be $\cH(2)\times \cH(2)$.) 
 \end{ex}

Our first lemma shows that irreducible affine invariant submanifolds do not contain any such strange sets.   

\begin{lem}\label{L:irredclosed}
Let $\cM$ be an irreducible affine invariant submanifold. Then any closed $GL(2, \bR)$ invariant subset of $\cM$ is a finite union of affine invariant submanifolds. 
\end{lem}

\begin{proof}
In an irreducible affine invariant submanifold, the ratio of areas of components is constant. By scaling factors, it can be assumed that the constants are one, and then the statement is given by multicomponent EMM. 
\end{proof}

\begin{lem}
Let $\cM$ be an irreducible affine invariant submanifold. Any flat subbundle of $T(\cM)$ not contained in $\ker(p)$ is equal to all of $T(\cM)$. 
\end{lem} 

\begin{proof}
This follows as in \cite[Theorem 5.1]{Wfield}. Irreducibility is only used so that the geodesic flow is hyperbolic. 
\end{proof}

\begin{lem}\label{L:irredspan}
Let $\cM$ be an irreducible affine invariant submanifold. Let $(S_1, S_2)$ be arbitrary finite sets. Then there is a dense open $GL(2, \bR)$ invariant subset of   $\cM$ on which the span of the $(S_1, S_2)$-recognizable twists contains $T(\cM)$. 
\end{lem}

\begin{proof}
Given the previous lemma,\ann{A: Added the words ``Given the previous lemma".} this follows from the proof of Proposition \ref{P:spanTM}, following Remark \ref{R:variant}. 
\end{proof}

\ann{A: Added some explanation of the strategy.} 
We now pause to outline the remainder of this subsection. Keeping in mind Example \ref{E:H2}, we cannot assume at this point that the smallest closed $GL(2,\bR)$ set containing $U$ is an affine invariant submanifold rather than a nonlinear object. We circumvent this issue by first working with the span of recognizable twists. Lemma \ref{L:multispan} gives a condition that can be used to show this span is at least as large as the tangent space to the smallest affine invariant submanifold containing $U$, and in Lemma \ref{L:noS1parallel} this condition is established almost everywhere on $U$. Finally, in the course of the proof of Proposition \ref{P:multimain}, it is shown that $U$ is an open subset of the span of the recognizable twists, which will allow us to conclude our analysis. 


\begin{lem}\label{L:multispan}
Let $\cM_1, \ldots, \cM_k$ be irreducible affine invariant submanifolds. Set $\cM=\cM_1\times \cdots \times \cM_k$.

Let $(X,\omega)\in \cM$, and assume that there is an $(X', \omega')$ in the orbit closure of $(X,\omega)$ with the following properties. 

First, no cylinder in any factor of $(X', \omega')$ has a parallel cylinder on any other factor with ratio of circumferences in $S_1$. Second, $\pi_i(X', \omega')$  has dense $GL(2, \bR)$ orbit in $\cM_i$ for each $i$. 

Then the span of the $(S_1, S_2)$-recognizable twists at $(X,\omega)$ contains $T(\cM)$. 
\end{lem}

Here $\pi_j: \cM\to \cM_j$\ann{A: Fixed typo.} is the projection into the $j$-th factor. By a cylinder on the $j$-th factor of $(X,\omega)$, we mean a cylinder on $\pi_j(X,\omega)$.

\begin{proof}
By Lemma \ref{L:irredspan}, the span of the $(S_1, S_2)$-recognizable twists contains $T(\cM)$ at $(X', \omega')$. Since the orbit of $(X, \omega)$ accumulates at $(X', \omega')$, by Lemma \ref{L:recnearby} we have that the same holds for $(X,\omega)$. 
\end{proof}

\begin{lem}\label{L:noS1parallel}
Let $\cM_1, \ldots, \cM_k$ be irreducible affine invariant submanifolds. Set $\cM=\cM_1\times \cdots \times \cM_k$.

Suppose $U$ is a  subset of $\cM$ that is equal to an open subset of a vector subspace in period coordinates, and suppose that $\cM$ is the smallest affine invariant submanifold containing $U$. Let $S_1\subset \bR$ be any finite set. Then for almost every surface $(X,\omega)$ in $U$, almost every point $(X', \omega')$ in the orbit closure of $(X,\omega)$ has the following properties. 

First, no cylinder in any factor of $(X', \omega')$ has a parallel cylinder on any other factor with ratio of circumferences in $S_1$. Second, $\pi_i(X', \omega')$  has dense $GL(2, \bR)$ orbit in $\cM_i$ for each $i$. 
\end{lem}

\begin{proof}
The case $k=1$ follows from Lemma \ref{L:irredclosed}. Assume $k=2$. 

Let $U_1$ be the subset of $(X,\omega)$ in  $U$ such that the orbit closure $\pi_i(X,\omega)$ is  $\cM_i$ for $i=1,2$. This set $U_1$ has full measure, because it is the complement of a countable number of linear subspaces (one for each affine invariant submanifolds of each $\cM_i$). 


Case 1: For almost every $(X, \omega)\in U$, the smallest affine invariant submanifold containing $(X, \omega)$ is reducible. The result follows, because for any such $(X,\omega) \in U_1$, the orbit closure of $(X,\omega)$ is the subset of $\cM_1\times \cM_2$ where the two components have the same ratio of area as $(X,\omega)$. \ann{A: Made cosmetic changes, because ``irreducible orbit closure" isn't defined, only irreducible affine invariant submanifold.}


Case 2: A positive measure subset  of $(X, \omega)\in U$ is contained in an irreducible affine invariant submanifold. Let $U_2$ be the subset of $U_1$ with this property. For  $(X, \omega)\in \cM_1\times \cM_2$, let $r(X,\omega)$ be the result of scaling the two factors so that they both have equal area. By the countability statement in multicomponent EMM, there is an affine invariant submanifold $\cN\subset \cM_1 \times \cM_2$ such that the following holds. For a positive measure set $U_3\subset U_2$, for every $(X,\omega)\in U_3$ the orbit closure of $r(X,\omega)$ is the subset of $\cN$ where the factors both have equal area.   Because $r^{-1}(\cN)$ is an analytic submanifold, since it contains a positive measure subset of $U$, it must in fact contain all of $U$. 


Note that the ratio of the area of the two components is a non-constant analytic function on $U$. (If it were constant, then $U$ would be contained in a scaling of $\cN$, showing that $\cM_1\times \cM_2$ is not the smallest affine invariant submanifold containing $U$.)  

 Let $\cN^{gen}$ be the set surfaces in  $\cN$  on which all parallel cylinders are in fact $\cN$-parallel.
  Note that $\cN^{gen}$ has full measure. Let $S_1'$ be the set of ratios of circumferences\ann{A: Added missing word.} of $\cN$-parallel cylinders on surfaces in $\cN$. This set is countable. (In fact, a generalized version of the Cylinder Finiteness Theorem applies here and gives that $S_1'$ is finite, but we only need countability here and that is easy.)
 
 Let $(X,\omega)$ be any surface in $U_3$ where the square root of the ratio of the areas of the two factors is not an element of $S_1$ divided by an element of $S_1'$. By the definition of $U_3$, the orbit closure of $(X, \omega)$ is $\cN$ with the components rescaled, and any surface $(X', \omega')$ in the $GL(2, \bR)$ orbit closure of $(X,\omega)$ with $r(X', \omega') \in \cN^{gen}$ has the desired property that no cylinder on $\pi_1(X', \omega')$ is parallel to a cylinder in $\pi_2(X', \omega')$ with ratio of circumferences in $S_1$.

Now, we will derive the result for $k>2$ from the result for $k=2$. For each $i,j$, let $U_{i,j}$ be the subset of $(X,\omega) \in U$ for which almost every surface $(X', \omega')$ in the orbit closure of $(X,\omega)$ has that $\pi_i(X', \omega')$ has dense $GL(2,\bR)$ orbit in $\cM_i$, and $\pi_j(X', \omega')$ has dense $GL(2,\bR)$ orbit in $\cM_j$, and  that no cylinder on $\pi_i(X', \omega')$ is parallel to a cylinder in $\pi_j(X', \omega')$ with ratio of circumferences in $S_1$. The lemma for $k=2$ gives that $U_{i,j}$ has full measure. The intersection of the $U_{i,j}$ gives the full measure set desired. 
\end{proof}

\begin{prop}\label{P:multimain} 
Suppose that $(X_n,\omega_n)\in \cM$ converge to $(X,\omega) \in \partial \cM$.

Then $(X, \omega)$ is contained in an affine invariant submanifold $\cM'$ whose tangent space at $(X,\omega)$ is given by, for infinitely many $n$, $\Ann(V)\cap T_{(X_n, \omega_n)}(\cM)\subset H^1(X_n,\Sigma_n, \bC)$ via the isomorphism between the tangent space to the boundary stratum at $(X,\omega)$ and $\Ann(V)\subset H^1(X_n,\Sigma_n, \bC)$.

Also, any finite sets $(S_1, S_2)$ that are $\cM$-adapted are also $\cM'$ adapted.  
\end{prop}

\begin{proof}
Define $U$ as in the proof of Proposition \ref{P:main}, and let $\cM'$ be the smallest affine invariant submanifold containing $U$. 

Pick $(S_1, S_2)$  that is $\cM$-adapted. 

By  Lemmas \ref{L:multispan} and \ref{L:noS1parallel}, for some  limit $(X', \omega')\in U$ the span of the $(S_1, S_2)$-recognizable twists contains $T(\cM')$. By Proposition \ref{P:rinside}, all these twists are contained in the span of the $(S_1, S_2)$-recognizable twists at $(X_n, \omega_n)+\xi_n$ for $n$ large enough, and hence we see that they are contained in $T(\cM)\cap \Ann(V)$. We get that $T(\cM')\subset T(\cM)\cap \Ann(V)$. 

The other containment proceeds just as in the proof of Proposition \ref{P:main}. 

Proposition \ref{P:rinside} also shows that $(S_1, S_2)$ is $\cM'$-adapted,  because any $(S_1, S_2)$-recognizable twist $u$ at a surface in $\partial \cM$ also gives a  $(S_1, S_2)$-recognizable twist at a nearby surface in $\cM$. Since $(S_1, S_2)$ is $\cM$-adapted, we have $u\in T(\cM)$. By definition, we have that $u\in \Ann(V)$. Hence $u\in T(\cM)\cap \Ann(V)=T(\cM')$.
\end{proof}

\subsection{Results using multicomponent EMM: Analogue of Theorem \ref{T:finite}}

Now we move on toward a version of Theorem \ref{T:finite}. The new danger in the multicomponent case is that a sequence of affine invariant submanifolds might not equidistribute to a larger affine invariant submanifold. For example, if $\cM_n$ is the locus in $\cH(2)\times \cH(2)$ where the second component is $n$ times the first, then the $\cM_n$ diverge in $\cH(2)\times \cH(2)$. 

Let $\cH_i, i=1, \ldots, k$ be strata of connected translation surfaces, and say $\cM$ is an affine invariant submanifold of $\cH_1\times\cdots\times\cH_k$. Let $A(\cM)$ be the set of $a\in \bR$ such that there are $1\leq i,j\leq k$, $i\neq j$ such that in $\cM$, the area of the $i$-th component is always $a$ times the area of the $j$-th component. So in the example above, $A(\cM_n)=\{n, 1/n\}$. 

\begin{lem}
For any finite sets $(S_1, S_2)$ and $d>0$, there is a finite set $A(S_1, S_2,d)\subset \bR$ such that the following holds. Suppose $\cM$ is an affine invariant submanifold  in a stratum of dimension at most $d$, and that $(S_1, S_2)$ is $\cM$-adapted. Then  $A(\cM)\subset A(S_1, S_2, d)$. 
\end{lem}

\begin{proof}
For simplicity, suppose $\cM$ consists of surfaces with two components. Pick any surface $(X, \omega)\in \cM$, and consider the set $\cC=\{C_1, \ldots, C_k\}$ of cylinders in some direction. Note $k<d$. (This is the only way $d$ will be used.)

Consider the equivalence relation on $\cC$ generated by $C_i$ equivalent to $C_j$ if the ratio of circumferences is in $S_1$. Let $\cC'\subset \cC$ be an equivalence class for this equivalence  relation. Note that the ratio of circumferences of any two cylinders in $\cC'$ is in a finite set depending only on $S_1$ and $d$.\ann{A: Typo fixed.}  Note also that $\cC'$ supports at least one $(S_1, S_2)$-recognizable twist, corresponding to $t_i=0$ for cylinders in $\cC\setminus \cC'$, and $t_i$ equal to the modulus for the cylinders in $\cC'$. 

Let $\alpha_i$ be the core curve of $C_i$.  Without loss of generality, assume $\cC'=\{C_1, \ldots, C_{k'}\}$. 

By definition, the $(S_1, S_2)$-recognizable twists in $\cC'$ are the twists of the form 
$$u=\sum_{i=1}^{k'} t_i c_i I_\Sigma(\alpha_i),$$ 
where the $t_i$ satisfy all linear equations with coefficients in $S_2$ that the moduli $m_i$ satisfy. There are only finitely many systems of linear equations in $S_2$ with at most $d$ variables. For each one of them that admits a solution $(t_1, \ldots, t_{k'})$ with all $t_i>0$, pick a solution $t_i>0$. 

Now, we use the fixed solution $(t_1, \ldots, t_{k'})$ to the set of linear equations defining the $(S_1, S_2)$-recognizable twists in $\cC$. Say that $C_1, \ldots, C_\ell$ are contained in the first component, and $C_{\ell+1}, \ldots, C_{k'}$ in the second. Then, applying the cylinder stretch arising from $u$ (an appropriate complex multiple of $u$ that changes area), we see that the area of the first component has derivative 
$\sum_{i=1}^\ell t_i c_i$, and the area of the second component has derivative $\sum_{i=\ell+1}^{k'} t_i c_i$. Since the ratios of areas is constant, this ratio must be $\sum_{i=1}^\ell t_i c_i$ divided by $\sum_{i=\ell+1}^{k'} t_i c_i$. 
\end{proof}

\begin{lem}
Let $A$ be a finite set, and let $\cH$ be a stratum of multicomponent translation surfaces. Suppose that $\cM_i, i=1,2, \ldots$, is an infinite sequence of distinct affine invariant submanifolds of fixed dimension $d$ of $\cH$ with $A(\cM_i)\subset A$. 

Then the closure of the union of the $\cM_i$ contains an affine invariant submanifold of dimension larger than $d$. 
\end{lem}

\begin{proof}
 Passing to a subsequence we may assume that all $\cM_i$ have the same pairs of components with the same fixed area ratios. In other words, $\cH$ is the product of strata of (possibly multicomponent) surfaces $\cH=\cH_1\times \cdots \times\cH_p$, and each $\cM_i$  is a product of irreducibles, $\cM_i=\cM_{i,1} \times \cdots \times \cM_{i,p}$. Here $\cM_{i,j}\subset \cH_j$, and the ratios of the components in each $\cM_{i,j}$ does not depend on $j$. 

For each $j$, since all $\cM_{i,j}$ live in the subset of $\cH_j$ where the ratios of areas are fixed independent of $i$, the closure of $\cup_j \cM_{i,j}$ is a finite union of affine invariant submanifolds. Let $\cN_j$ be the component of largest dimension. 

The union of the closure of the $\cM_i$ contains $\cN_1\times \cdots\times \cN_p$, and it is easy to see that this has dimension larger than $d$. 
\end{proof}

\begin{proof}[Proof of Theorem \ref{T:main3}]
Given Proposition \ref{P:multimain}, to conclude the proof of Theorem \ref{T:main3} it suffices to prove a version of Theorem \ref{T:finite} in the multicomponent situation. 

So we assume to the contrary, that for some fixed boundary stratum $\cH'$, there are infinitely many affine invariant submanifolds $\cM_i$ arising as in Proposition \ref{P:multimain}. Pass to a subsequence where the $\cM_i$ have maximal dimension $d$. The previous two lemmas give that the closure of the union of the $\cM_i$ contains an affine invariant submanifold of dimension larger than $d$, and so we can conclude the proof exactly as in the proof of Theorem \ref{T:finite}. 
\end{proof}

\section{Basic boundary theory for strata}\label{S:Basic}

In the first subsection we prove Propositions \ref{P:fn} and \ref{P:Vn}. In the second we prove Proposition \ref{P:AnnV}. In the third, we show that the notion of convergence defined in Section 2 can be rephrased in terms of the Deligne-Mumford compactification of $\cM_{g,s}$. 

\subsection{Proof of Proposition \ref{P:Vn}}

%
%

The next lemma considers a stratum  $\cH$ in the boundary of some other stratum $\cH'$. 

\begin{lem}\label{L:Vthin}
Let $\cH'$ and $\cH$ be two strata of multicomponent translation surfaces, and let $(X,\omega, \Sigma)\in \cH$. Let $U$ be small neighborhood of $\Sigma$, such that  each component of $U$ is homeomorphic to a disk. Fix $\delta>0$ and a neighborhood $\cO$ of $\omega$ in the space of differential 1-forms on $X\setminus U$. Let $\cU$ be the neighborhood in $\cH'$ of $(X, \omega,\Sigma)$ that consists of $(X',\omega', \Sigma')\in \cH'$ such that there is a map $g:X\setminus U\to X'$ that is a diffeomorphism onto its range, such that
\begin{enumerate}
\item $g^*(\omega')\in \cO$ and 
\item the injectivity radius of points not in the image of $g$ is at most $\delta$. 
\end{enumerate}
Let $\Delta$ be the length of the shortest saddle connection on $(X,\omega, \Sigma)$. Then for $U, \delta, \cO$ small enough, any saddle connection on $(X', \omega', \Sigma')$ of length less than $\Delta/2$ is contained in the $\delta$-thin part (the subset with injectivity radius at most $\delta$). 
\end{lem}

\begin{proof}
We may pick $U$ to be a union of small balls about points of $\Sigma$ small enough that any path between different components of $U$ on $(X,\omega, \Sigma)$ has length at least $\frac3{4}\Delta$. If $\cO$ is chosen small enough, then any path between different components of $X'\setminus g(X\setminus U)$ has length at least $\Delta/2$, which gives the result since $X'\setminus g(X\setminus U)$ is contained in the $\delta$-thin part by supposition. 
\end{proof}

\begin{proof}[Proof of Propositions \ref{P:fn} and \ref{P:Vn}]
We may assume that the $U_n$ are homeomorphic to disks and the $U_{n-1}\setminus U_n$ are homeomorphic to annuli. 

Given $g_n: X\setminus U_n\to X_n$, define the collapse  map $f_n$ to be the inverse of $g_n$ on $X\setminus g_n(U_{n-1})$, and to map $X_n\setminus g_n(X\setminus U_{n})$ to the point $\Sigma\cap U_n$. The remainder of $X_n$ is a collection of annuli on which $f_n$ can be defined to interpolate.  

By Lemma \ref{L:Vthin}, for all $\delta>0$ small enough and $n$ large enough, $$V_n=\ker(f_n:H_1(X_n, \Sigma_n, \bZ)\to H_1(X, \Sigma, \bZ))$$ is spanned by the relative homology of the $\delta$-thin part\ann{A: Added sentence in parentheses.}. (Indeed, the set $X_n\setminus g_n(X\setminus U_{n})$ that is collapsed can be assumed to be in the thin part, and Lemma \ref{L:Vthin} says that the thin part isn't larger.)  Also for $n$ large enough, $V_n$ is locally constant: there is a neighborhood $\cU$ of $(X, \omega, \Sigma)$ given by Lemma \ref{L:Vthin} that contains all $(X_n, \omega_n, \Sigma_n)$ for $n$ sufficiently large, on which the homology of the $\delta'$-thin part is invariant under parallel transport.  
\end{proof}

\subsection{Proof of Proposition \ref{P:AnnV}: How strata fit together} 


Suppose that $(X_n, \omega_n, \Sigma_n)\in \cH'$  converges to $(X,\omega, \Sigma)\in \cH$, and let $V\subset H_1(X_n, \Sigma_n, \bZ)$ be the space of vanishing cycles.\ann{A: Some significant expositional changes have been made to Section 9.2. They do not effect the rest of the paper. The purpose of the chances was to clarify the proofs. } 

\begin{lem}
$\Ann(V)$ is naturally identified with the tangent space of the boundary stratum $\cH$ via  the collapse maps $f_n$. 
\end{lem}

\begin{proof}
By the definition of $V$ as the kernel of the action of the collapse map on relative homology, and because relative cohomology is the dual to relative homology,  the induced map $f_n^*$ on relative cohomology is an isomorphism from $H^1(X, \Sigma, \bC)$ to $\Ann(V)$. 
\end{proof}

The remainder of Proposition \ref{P:AnnV} is given by the following. 

\begin{prop}\label{P:plusv}
There exists a neighborhood of 0 in $\Ann(V)$ and an $n_0>0$ such that for $\xi, \xi_n$ in this neighborhood with $\xi_n\to \xi$, the paths  $(X_n, \omega_n, \Sigma_n)+t\xi_n$ and $(X, \omega, \Sigma)+t\xi$ are well defined for $t\in [0,1]$ and $n>n_0$, and 
 $(X_n, \omega_n, \Sigma_n)+\xi_n\to (X, \omega, \Sigma)+\xi$.
\end{prop}

It is plausible that part of this (that the paths are well defined) could be derived using the AGY norm \cite{AGY}, in particular using \cite[Proposition 5.5]{AG}.
Here we will instead sketch a proof using Cech cohomology. 

The use of Cech cohomology to describe the tangent space to Teichm\"uller space is standard and is sometimes called Kodaira-Spencer deformation theory. See, for example, \cite[Section 7.2.4]{IT} for an introduction.  Cech cohomology has also been previously used to describe deformations of translation surfaces. See, for example, \cite[Section 2]{HubbMas} and \cite{McM:nav}.  

Given a finite open cover $\cU=\{U_1, \ldots, U_p\}$ of a translation surface  $(X, \omega, \Sigma)$, we consider 1-cocycles $w$  of $\cU$ with coefficients in the sheaf of locally constant functions. Such a $w$ consists of a complex number  $w_{i,j}\in \bC$ for each $1\leq i, j\leq p$ with $U_i\cap U_j\neq \emptyset$. These numbers are required to satisfy $w_{i,j}=-w_{j,i}$ for all $i,j$, and $w_{i,k}=w_{i,j}+w_{j,k}$ for all $i,j,k$ with $U_i\cap U_j\cap U_k\neq \emptyset$. We define $\|w\|=\max |w_{i,j}|$. 

We consider deformations of $(X, \omega, \Sigma)$ defined by sliding each piece $U_j$ over the piece $U_i$ in direction $w_{i,j}$. See Figure \ref{F:wij}.
\begin{figure}[h]
\includegraphics[width=0.8\linewidth]{wij.pdf}
\caption{}
\label{F:wij}
\end{figure}
Such a deformation is not always well defined. For example,  as the pieces slide more and more over each other, a zero of $\omega$ can hit the boundary of some $U_{i,j}$, preventing that $U_{i,j}$ from being slid any further (as in Figure \ref{F:maxT}, left); or, if pieces are being slid away from each other, a ``hole" in the surface may appear (as in Figure \ref{F:maxT}, right). 
\begin{figure}[h]
\includegraphics[width=0.8\linewidth]{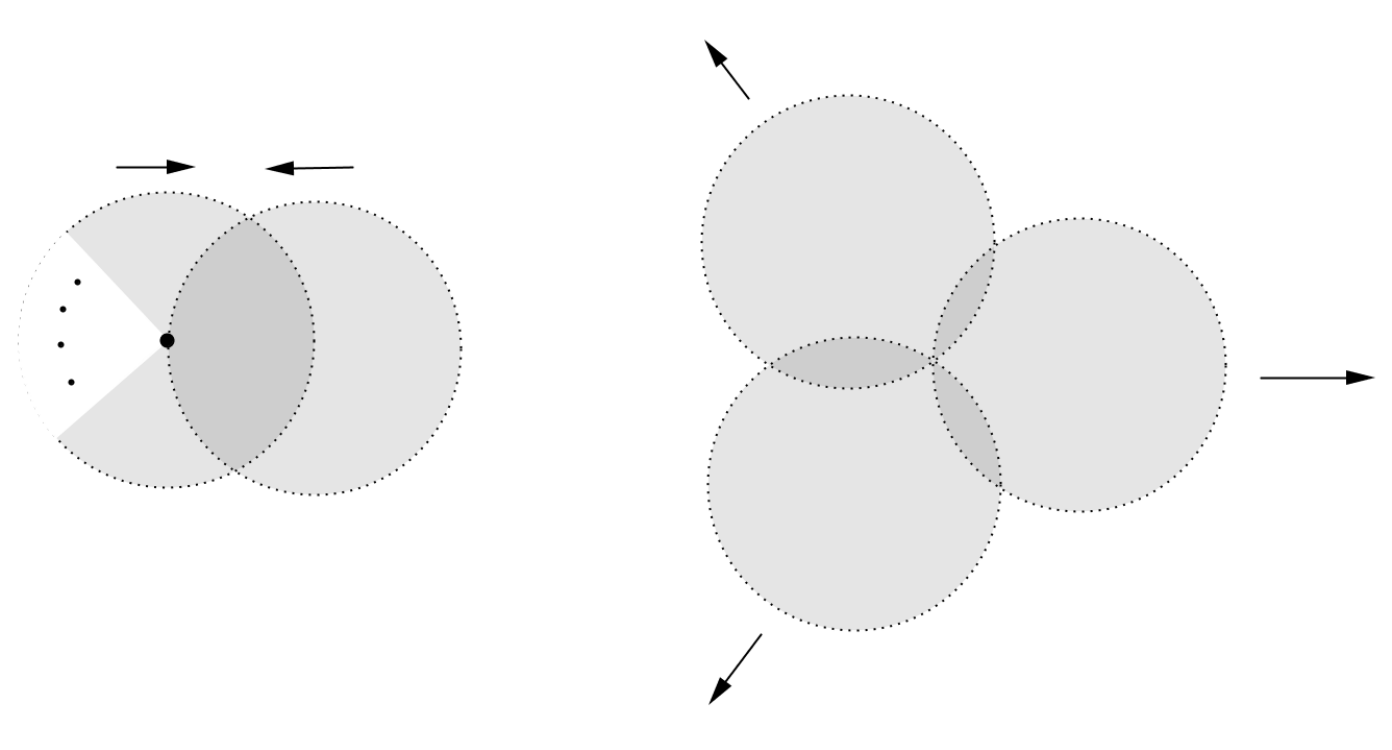}
\caption{}
\label{F:maxT}
\end{figure}

The following definition and lemma provide a precise definition of such deformations and a proof that they are well defined when $w$ is small enough and the open cover is nice enough.

\ann{A: This definition is new.}
\begin{defn}\label{D:decent}
Define a finite open cover  $\cU=\{U_1, \ldots, U_p\}$ of a translation surface $(X, \omega, \Sigma)$ to be decent if each point of $\Sigma$ is contained in exactly one  $U_i$ and is disjoint from the closures of every other $U_j$, and  the closures of $U_i$ and $U_j$ are disjoint whenever $U_i$ and $U_j$ are disjoint. Define $\cU$ to be $\e$ decent if the following conditions hold. 
\begin{enumerate}
\item If $U_i$ and $U_j$ are disjoint, the distance between them is greater than $2\e$.
\item Every point of $\Sigma$ is contained in a unique $U_i$, and its distance to every other $U_j$ is greater than $2\e$.
\item For every point $x\in X$, there is some $U_i$ that contains all points of distance at most $\e$ from $U_i$.
\end{enumerate} 
\end{defn} 

Note that every decent cover is $\e$ decent for some $\e>0$, and that the definition allows a $U_i$ to contain more than one point of $\Sigma$.

If $x$ is a point on $(X, \omega, \Sigma)$ not in $\Sigma$, and $e\in \bC$, then $x+e$ is defined to be the result of straight line flow in direction $e/|e|$ for time $|e|$; if the trajectory of straight line flow hits a point of $\Sigma$ before time $|e|$, then  $x+e$ is  undefined. 

\begin{lem}\label{L:exists}
Let $\cU=\{U_1, \ldots, U_p\}$ be an $\e$ decent open cover of  $(X, \omega, \Sigma)$. Let $w$ be a 1-cocycle  of $\cU$ with coefficients in the sheaf of locally constant functions, and assume $\|w\|<\e$. 

Consider the disjoint union of the $U_i$ with the following identifications.  A point $x_i\in U_i$ is identified to $x_j\in U_j$ if  $U_i\cap U_j\neq \emptyset$ and $x_i+w_{i,j}$ is  defined and equal to $x_j$. This defines a translation surface in the same stratum as $(X, \omega, \Sigma)$.

Furthermore, the path of translation surfaces corresponding to $tw, t\in [0,1]$, is linear in period coordinates. 
\end{lem}\ann{A: This lemma  replaces a less formal discussion of the same material.} 

Note that the lemma allows the $U_i$ to be rather arbitrary: they may have complicated boundaries, and neither the $U_i$ nor their intersections are assumed to be simply connected or even connected.  

\begin{proof}
Say that  $x_i\in U_i$ is identified to $x_j\in U_j$ if  $U_i\cap U_j\neq \emptyset$ and $x_i+w_{i,j}$ is  defined and equal to $x_j$. (By definition, this implies there is a straight line segment from $x_i$ to $x_j$ on $(X, \omega, \Sigma)$ with holonomy $w_{i,j}$. This straight line segment is not required to stay in $U_i\cup U_j$.) A priori, there is the possibility that $x_i\in U_i$ and $x_j\in U_j$ could have the same image in the quotient by these identification without being identified to each other. This will ruled out shortly when we show the relationship ``is identified to" is transitive. 

Suppose $x\in U_i$ has distance less than $\|w\|$ to a point of $\Sigma$. By Definition \ref{D:decent}(2) this point of $\Sigma$ lies in $U_i$, and the distance between $x$ and any $U_j, j\neq i$ is greater than $\|w\|$. Hence the $\|w\|$ neighborhood of every point of $\Sigma$ is contained in a single $U_i$, and no point of this neighborhood is identified to any point of any $U_j, j\neq i$. Note also that for any point $x$ of distance at least $\|w\|$ from $\Sigma$, $x+e$ is defined whenever $|e|\leq \|w\|$. 

Suppose $x_i\in U_i$ is identified to  $x_j=x_i+w_{i,j}\in U_j$, and $x_j$ is identified to $x_k=x_j+w_{j,k}\in U_k$. Then the distance from $U_i$ to $U_k$ is at most $|w_{i,j}|+|w_{j,k}|\leq 2\|w\|$, so $U_i\cap U_k\neq \emptyset$ by Definition \ref{D:decent}(1). Consider the flat geodesic $\gamma$ from $x_i$ to $x_j$ homotopic rel endpoints to the concatenation of the straight line flow path from $x_i$ to $x_{j}$  followed by the straight line flow path from $x_j$ to $x_k$.  By the previous comment, both $x_i$ and $x_k$ have distance greater than $\|w\|$ from $\Sigma$. Since $\gamma$ has length at most $2\|w\|$, this gives that $\gamma$ is a straight line segment from $x_i$ to $x_k$. This shows that $x_k=x_i+w_{i,k}$, so we conclude that $x_i$ is identified to $x_k$. 

Gluing pieces of translation surfaces via translations will always yield a translation surface, so it remains only to check that this translation surface is closed (unlike Figure \ref{F:maxT}, right). By   Definition \ref{D:decent}(3), for each point $x$ in the boundary of a $U_i$, there is some $U_j$ that contains all points on $(X, \omega, \Sigma)$ of distance at most $\|w\|$ from $x$. After the identifications described, this $U_j$ will be glued to $U_i$ in such a way to cover the point $x$ on the boundary of $U_i$. This shows the identification gives a closed surface. 

To verify that the path of translation surfaces obtained is linear, it suffices to show that if  $U_i\cap U_j\neq \emptyset$, then for any path joining a point of $\Sigma \cap U_i$ with a point of $\Sigma\cap U_j$ the period of this path varies linearly. This is the case by definition. 
\end{proof}

\begin{lem}
For each translation surface $(X,\omega, \Sigma)$, there is a decent open cover such that the derivatives of deformations of $(X,\omega, \Sigma)$ that arise as above from 1-cocycles on this cover span the tangent space to the stratum at $(X,\omega, \Sigma)$. This open cover may be chosen so that all points of $\Sigma$ are contained in a single element of the open cover. 
\end{lem} \ann{A: This lemma has changed a bit.}

The proof follows from general theory, but we give a sketch of a concrete approach. 

\begin{proof}
We begin with the first statement.
Pick a maximal spanning subtree of a triangulation of the surface. A cover given by neighborhoods of half edges, and the complement of a neighborhood of the spanning tree, is easily seen to work.   The periods of the edges in the spanning tree can be changed independently by deformations given by cocycles.

To arrange for all elements of $\Sigma$ to be contained in a single element of the open cover, we may modify the cover  $\cU=\{U_1, \ldots, U_p\}$ as follows. We use the fact that  each element of $\Sigma$ is contained in a single $U_i$ and is disjoint from the $\overline{U_j}, j\neq i$.  
 First, add an open set $U_{0}'$ that consists of  points of distance less than $2r$ from $\Sigma$, where $r>0$ is very small. Then remove from each  $U_i$  the set of distance at most $r$ from $\Sigma$ to get a new open set $U_i'$. Then $\{U_0', U_1', \ldots, U_p'\}$ will be the new open cover.

For any cocycle $w$ supported on the original open cover, we can find a new cocycle $w'$ supported on the new open cover by $w'_{i,j}=w_{i,j}$ if $i,j\neq 0$ and $w'_{0,i}=w'_{i,0}=0$. The cocycles $w$ and $w'$ describe exactly the same deformation. 
\end{proof}

\begin{lem}
Fix an $\e$ decent  cover of $(X,\omega, \Sigma)$, such that the derivatives of deformations of $(X,\omega, \Sigma)$ that arise as above from 1-cocycles on this cover span the tangent space to the stratum at $(X,\omega, \Sigma)$, and so that $\Sigma$ is contained in a single element of the cover.
For $n$ sufficiently large, the pull back of this cover to $(X_n, \omega_n, \Sigma_n)$   is $\e/2$ decent, and the deformations given by 1-cocycles supported on this cover span $\Ann(V)$. 
\end{lem} 

Recall from the discussion after Proposition \ref{P:fn} that the maps $f_n:X_n\to X$ are only defined after identifying some points of $\Sigma$ with each other. The reason we have required all points of $\Sigma$ to lie in a single element of the open cover is so that we obtain an open cover even after points of $\Sigma$ are identified with each other. 

\begin{proof}
The first claim follows because, for $n$ large enough, the map $g_n$ given by Definition \ref{D:converge} is close to an isometry for the flat metric. 

Let $U$ be the union of the balls of tiny radius  centered at points of $\Sigma$. 
We require the radius to be smaller than $\e/2$. 
As in the proof of Lemma \ref{L:Vthin}, we see that $V$ is supported on $X_n\setminus g_n(X\setminus U)$, which is contained in a single element of the open cover. Hence the (derivatives of) deformations described by cocycles supported on this cover are in $\Ann(V)$. 

The dimension of the boundary stratum is the dimension of $\Ann(V)$. The relative homology  of $(X, \Sigma)$ is isomorphic to $H_1(X_n, \Sigma_n)/V$, and the effect of the deformation given by a 1-cocycle on period coordinates is given by this isomorphism. In particular, linearly independent 1-cocycles on $X$ give rise to linearly independent 1-cocycles on $X_n$.  

Thus  the cocycles supported on the cover of $X_n$ give rise to a linearly independent set of deformations of dimension equal to the dimension of $\Ann(V)$, and we conclude that these deformations span $\Ann(V)$. 
\end{proof}

\begin{proof}[Proof of Proposition \ref{P:plusv}]\ann{A: This proof has been rewritten.} 
Fix an $\e$ decent cover $\{U_1, \ldots, U_p\}$ of the limit surface $(X,\omega, \Sigma)$ such that the derivatives of deformations of $(X,\omega, \Sigma)$ that arise as above from 1-cocycles on this cover span the tangent space to the stratum at $(X,\omega, \Sigma)$, and so that $\Sigma$ is contained in a single element of the cover. Pulling back the cover by $f_n$ gives an $\e/2$ decent cover of  $(X_n, \omega_n, \Sigma_n)$. By the previous lemmas, deformations on this cover span in $\Ann(V)$.

There is a neighborhood of $0$ in $\Ann(V)$ given by deformations obtained from cocycles of norm at most $\e/2$. If $\xi, \xi_n$ are in this neighborhood, we have that  $(X_n,\omega_n, \Sigma_n)+\xi_n$ and  $(X,\omega, \Sigma)+\xi$ are all well defined. It remains to show convergence. We will do this assuming $\xi_n, \xi$ are described by  cocycles of norm at most $\e'$, where $0<\e'<\e/2$ will be specified later. 

The complement of the image of $g_n$ is a set of points with small injectivity radius that may have non-trivial topology. We call the complement of the image of $g_n$ the small subsurface of $(X_n,\omega_n, \Sigma_n)$. (There are at most $|\Sigma|$ many connected components of the small subsurface, and there may be  fewer.)

$(X_n,\omega_n, \Sigma_n)+\xi_n$ is  obtained by gluing the disjoint union of the open sets $f_n^{-1}(U_k)$, with the change in gluing compared to $(X_n,\omega_n, \Sigma_n)$  described by $\xi_n$. Since the change in gluing does not effect the metric on each element of the cover, and since the small surface is contained in a single element of the cover, it remains in the set of points with small injectivity radius on $(X_n,\omega_n, \Sigma_n)+\xi_n$. 

Assume $\e'<\e/10$. 
Then any sufficiently short saddle connection (say, length at most $\e/100$) on $(X_n,\omega_n, \Sigma_n)+\xi_n$ must be contained in one of these subsurfaces. 

To get from $(X,\omega, \Sigma)+\xi$ minus a neighborhood of $\Sigma$ to $(X_n,\omega_n, \Sigma_n)+\xi_n$ minus the union of the small subsurfaces, it suffices to change the flat metric very slightly on each element of the open cover ($g_n$ is close to an isometry for the flat metric), and then change the gluing instructions very slightly ($\xi_n$ is close to $\xi$). Hence there should be a map between these two spaces that is almost a flat isometry, and this map verifies the convergence of  $(X,\omega, \Sigma)+\xi$ to $(X_n,\omega_n, \Sigma_n)+\xi_n$.

We now sketch one method of constructing such a map. Fix a triangulation of $(X,\omega, \Sigma)$ with  vertices  $\Sigma$. Pick $\e'$ small enough so that this triangulation remains a triangulation on $(X,\omega, \Sigma)+\xi$ whenever $\|\xi\|<\e'$, that is, all the saddle connections of the triangulation remain saddle connections on $(X,\omega, \Sigma)+\xi$. 

Fix $\delta>0$ be very small, in particular much smaller than the the length of the smallest edge in the triangulation. For each edge of the triangulation of $(X,\omega ,\Sigma)$, consider the two points that are distance $\delta$ from the endpoints. Fix a given triangle in the triangulation, and denote these points $A,B,C,D,E,F$ as in Figure \ref{F:Hexagons} right. These six points are the corners of a hexagon $H$ that shares three edges with the triangle. (One should consider all triangles simultaneously, but for notational simplicity we will focus on just one.)

For $n$ large enough, all such hexagons are contained in the domain of $g_n$. Furthermore, for $n$ large enough, there is a hexagon $H_n$ spanned by the points $g_n(A), g_n(B), g_n(C), g_n(D), g_n(E), g_n(F)$. 
The side lengths of  $H_n$ converge to those of $H$. 

\begin{figure}[h]
\makebox[\textwidth][c]{
\includegraphics[width=1.05\linewidth]{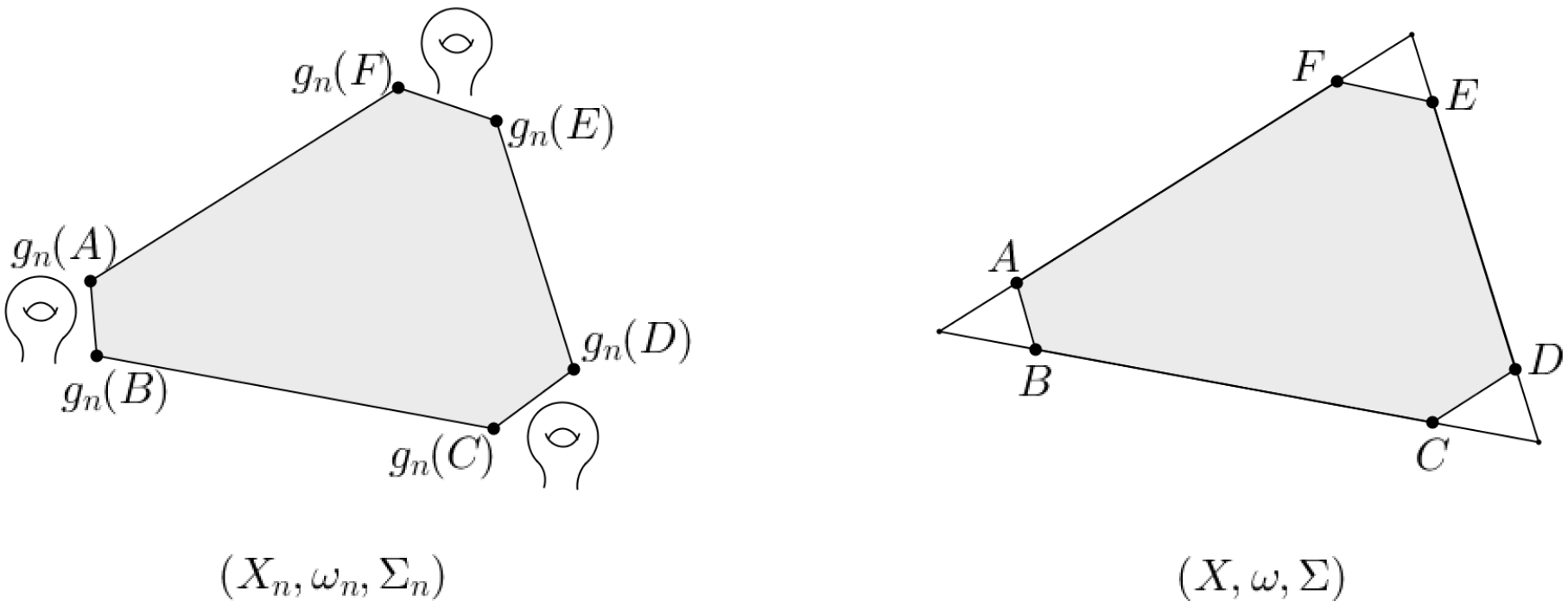}}
\caption{}
\label{F:Hexagons}
\end{figure}

We may assume that each point on an edge of the triangulation that is distance $\delta$ from a vertex is contained in only one element of the open cover, because $\delta$ is very small and because of Definition \ref{D:decent}(1). Each  set of the open cover minus a neighborhood of $\Sigma$ maps continuously into both $(X, \omega, \Sigma)+\xi$ and $(X_n, \omega_n, \Sigma_n)+\xi_n$, so  we get corresponding points $A', B', C', D', E', F'$ on $(X, \omega, \Sigma)+\xi$ and $A_n', B_n', C_n', D_n', E_n', F_n'$ on $(X_n, \omega_n, \Sigma_n)+\xi_n$. For $n$ large enough, these points span hexagons  $H'$ and $H_n'$. As $n\to\infty$, the sides lengths of $H_n'$ converge to those of $H'$. 

Pick a point $P$ in the interior of $H'$. Joining $P$ to the six points $A', B'$ etc., we triangulate $H'$. We can find a piecewise linear map from $H'$ to $H'_n$ that is linear on each triangle, and that has derivative closer and closer to the identity as $n\to\infty$. 

We now have for any fixed $\delta$ small enough a sequence of piecewise linear maps from $(X, \omega, \Sigma)+\xi$ minus a   neighborhood of $\Sigma$ depending on $\delta$ to $(X_n, \omega_n, \Sigma_n)+\xi_n$ that is closer and closer to preserving the flat metric as $n\to\infty$. Using the fact that the complement of the image of $g_n$ has injectivity radius going to zero as $n\to \infty$, we may see that for $\delta$ small enough and $n$ large enough, the injectivity radius of the complement of the image of these piecewise linear maps is arbitrarily small. The piecewise linear maps can be approximated by smooth maps that are also close to preserving the flat metric, and a sequence can be chosen along which $\delta\to0$ and $n\to \infty$ to give the sequence required by Definition \ref{D:converge} to verify convergence. 
\end{proof}

\subsection{Connection to Deligne-Mumford}

Here we sketch the proof of the following result. It may be familiar to some experts, and is not used in the paper. The sketch is included for completeness. 

Let $\overline{\cM}_{g,s}$ be the Deligne-Mumford compactification of the moduli space $\cM_{g,s}$ of genus $g$ Riemann surfaces with $s$ marked points. Recall that given a surface  $(X, \omega, \Sigma)$, the points of $\Sigma$ are the marked points, and $\Sigma$ is assumed to include the zeros of $\omega$. 

\begin{prop}\label{P:D-M}
$(X_n, \omega_n, \Sigma_n)$ converges to $(X, \omega, \Sigma)$ in the sense of Section 2 if and only if, for each limit point of the sequence in the bundle of Abelian differentials over $\overline{\cM}_{g,s}$, the result of removing the components on which the Abelian differential is zero is $(X,\omega, \Sigma)$. 
\end{prop}

 To be precise, one should remove the components on which the Abelian differential is zero, and the nodes, and then add back in finitely many marled\ann{A: Word ``marked" added.} points as required to get a closed surface. 

\begin{proof}
We will  first assume that $(X_n, \omega_n, \Sigma_n)$ converges in the bundle of Abelian differentials, say to $(X', \omega', \Sigma')$, and argue that this implies convergence as in Section 2. Let $U_n$ be  decreasing neighborhoods of the nodes on $(X', \omega', \Sigma')$, such that $\cap U_n$ is  the set of nodes.

We may assume that there is a map $g_n$ from $X'\setminus U_n$ to $X_n$ such that 
$$\lim_{n\to \infty}\omega-g_n^*(\omega_n)=0.$$

We now claim that the injectivity radius of any point in $X_n\setminus g_n(X'\setminus U_n)$ goes to zero (with respect to the flat metric given by $\omega_n$). This claim follows as in \cite[Lemma 4.3]{McM:nav}: An  isoperimetric inequality \cite[Theorem A.2]{McM:nav} gives that, since the annuli in $X_n\setminus g_n(X'\setminus U_n)$ have small boundary, almost all of their area is contained in a cylinder of small circumference. Points in a cylinder of small circumference have small injectivity radius.   The boundary of $X_n\setminus g_n(X'\setminus U_n)$ is small because the map $g_n$ is almost isometric.

\begin{rem}
See also \cite[Theorem 2]{CS} for the statement of a more precise isoperimetric inequality, not required here. The work of Rafi  \cite{Rafi}, \cite[Section 4]{EKZbig} provides an alternate approach. 
\end{rem}

Define $G_n$ as the union of the components of $X'$ on which $\omega'$ is nonzero, minus $U_n$. 

We now claim that every point of $X_n \setminus g_n (G_n)$ has injectivity radius going to zero as $n$ goes to infinity. Given the first claim  it suffices to show this on the image under $g_n$ of components of $X'$ on which $\omega'$ is zero. However, $g_n$ does not distort the flat metric much, so the flat metric is uniformly small on this image. This gives the claim. 

Define $(X,\omega, \Sigma)$ to be $(X', \omega', \Sigma')$ with the components where $\omega'$ is 0 removed. Convergence $(X_n, \omega_n, \Sigma_n)\to (X,\omega, \Sigma)$ using the definition in Section 2 follows because the map $g_n$ on $(X', \omega', \Sigma')$ restricted to $(X,\omega, \Sigma)$ gives the map $g_n$ referred to in Section 2. 

To show the converse, it suffices to consider  $(X_n, \omega_n, \Sigma_n)$ that converge as in Section 2 to $(X,\omega, \Sigma)$. Without loss of generality, assume they converge to $(X', \omega', \Sigma')$ in the bundle of Abelian differentials over Deligne-Mumford. 

Let $(X'',\omega'', \Sigma'')$ be the result of removing the zero area components from $(X', \omega', \Sigma')$. We need to show that  $(X'',\omega'', \Sigma'')$ is equal to $(X,\omega, \Sigma)$. This is straightforward, because the composition of the map from $(X'',\omega'', \Sigma'')$ to $(X_n, \omega_n, \Sigma_n)$ with the collapse map to $(X,\omega, \Sigma)$ gives maps from a complement of a neighborhood of $\Sigma''$ in $(X'',\omega'', \Sigma'')$ to  $(X,\omega, \Sigma)$ whose limits in the compact open topology send $\omega$ to $\omega''$. 
\end{proof}

\bibliography{mybib}{}

\newcommand{\etalchar}[1]{$^{#1}$}
\providecommand{\bysame}{\leavevmode\hbox to3em{\hrulefill}\thinspace}
\providecommand{\MR}{\relax\ifhmode\unskip\space\fi MR }
\providecommand{\MRhref}[2]{%
  \href{http://www.ams.org/mathscinet-getitem?mr=#1}{#2}
}
\providecommand{\href}[2]{#2}
\begin{thebibliography}{McM06b}

\bibitem[AD16]{AD}
Artur Avila and Vincent Delecroix, \emph{Weak mixing directions in
  non-arithmetic {V}eech surfaces}, J. Amer. Math. Soc. \textbf{29} (2016),
  no.~4, 1167--1208.

\bibitem[AG13]{AG}
Artur Avila and S{\'e}bastien Gou{\"e}zel, \emph{Small eigenvalues of the
  {L}aplacian for algebraic measures in moduli space, and mixing properties of
  the {T}eichm\"uller flow}, Ann. of Math. (2) \textbf{178} (2013), no.~2,
  385--442.

\bibitem[AGY06]{AGY}
Artur Avila, S{\'e}bastien Gou{\"e}zel, and Jean-Christophe Yoccoz,
  \emph{Exponential mixing for the {T}eichm\"uller flow}, Publ. Math. Inst.
  Hautes \'Etudes Sci. (2006), no.~104, 143--211.

\bibitem[AMY13]{AMY}
Artur Avila, Carlos Matheus, and Jean-Christophe Yoccoz,
  \emph{{$SL(2,\Bbb{R})$}-invariant probability measures on the moduli spaces
  of translation surfaces are regular}, Geom. Funct. Anal. \textbf{23} (2013),
  no.~6, 1705--1729.

\bibitem[AN]{AN}
David Aulicino and Duc-Manh Nguyen, \emph{Rank two affine submanifolds in
  $\mathcal{H}(2,2)$ and $\mathcal{H}(3,1)$}, preprint, arXiv:1501.03303
  (2015).

\bibitem[ANW16]{ANW}
David Aulicino, Duc-Manh Nguyen, and Alex Wright, \emph{Classification of
  higher rank orbit closures in {$\mathcal{H}^{\rm{odd}}(4)$}}, J. Eur. Math.
  Soc. (JEMS) \textbf{18} (2016), no.~8, 1855--1872.

\bibitem[Api]{Apisa}
Paul Apisa, \emph{${GL}_2(\mathbb{R})$ orbit closures in hyperelliptic
  components of strata}, preprint, arXiv:1508.05438 (2015).

\bibitem[Bai07]{Ba}
Matt Bainbridge, \emph{Euler characteristics of {T}eichm\"uller curves in genus
  two}, Geom. Topol. \textbf{11} (2007), 1887--2073.

\bibitem[BCG{\etalchar{+}}]{Many}
Matt Bainbridge, Dawei Chen, Quentin Gendron, Samuel Grushevsky, and Martin
  Moeller, \emph{Compactification of strata of abelian differentials},
  preprint, arXiv:1604.08834 (2016).

\bibitem[BCR98]{BCR}
Jacek Bochnak, Michel Coste, and Marie-Fran{\c{c}}oise Roy, \emph{Real
  algebraic geometry}, Ergebnisse der Mathematik und ihrer Grenzgebiete (3)
  [Results in Mathematics and Related Areas (3)], vol.~36, Springer-Verlag,
  Berlin, 1998, Translated from the 1987 French original, Revised by the
  authors.

\bibitem[Boi15]{Boissy}
Corentin Boissy, \emph{Connected components of the strata of the moduli space
  of meromorphic differentials}, Comment. Math. Helv. \textbf{90} (2015),
  no.~2, 255--286.

\bibitem[Cal04]{Ca}
Kariane Calta, \emph{Veech surfaces and complete periodicity in genus two}, J.
  Amer. Math. Soc. \textbf{17} (2004), no.~4, 871--908.

\bibitem[CE15]{EsCh}
Jon Chaika and Alex Eskin, \emph{Every flat surface is {B}irkhoff and
  {O}seledets generic in almost every direction}, J. Mod. Dyn. \textbf{9}
  (2015), no.~01, 1--23.

\bibitem[Che]{Chen}
Dawei Chen, \emph{Degenerations of {A}belian {D}ifferentials}, preprint,
  arXiv:1504.01983 (2015).

\bibitem[CS09]{CS}
Jaigyoung Choe and Richard Schoen, \emph{Isoperimetric inequality for flat
  surfaces}, Proceedings of the 13th {I}nternational {W}orkshop on
  {D}ifferential {G}eometry and {R}elated {F}ields [{V}ol. 13], Natl. Inst.
  Math. Sci. (NIMS), Taej\u on, 2009, pp.~103--109.

\bibitem[EKZ14]{EKZbig}
Alex Eskin, Maxim Kontsevich, and Anton Zorich, \emph{Sum of {L}yapunov
  exponents of the {H}odge bundle with respect to the {T}eichm\"uller geodesic
  flow}, Publ. Math. Inst. Hautes \'Etudes Sci. \textbf{120} (2014), 207--333.

\bibitem[EM]{EM}
Alex Eskin and Maryam Mirzakhani, \emph{Invariant and stationary measures for
  the {$SL(2,\bR)$} action on moduli space}, preprint, arXiv:1302.3320 (2013).

\bibitem[EM01]{EMa}
Alex Eskin and Howard Masur, \emph{Asymptotic formulas on flat surfaces},
  Ergodic Theory Dynam. Systems \textbf{21} (2001), no.~2, 443--478.

\bibitem[EMM15]{EMM}
Alex Eskin, Maryam Mirzakhani, and Amir Mohammadi, \emph{Isolation,
  equidistribution, and orbit closures for the {${\rm SL}(2,\Bbb R)$} action on
  moduli space}, Ann. of Math. (2) \textbf{182} (2015), no.~2, 673--721.

\bibitem[EMZ03]{EMZboundary}
Alex Eskin, Howard Masur, and Anton Zorich, \emph{Moduli spaces of abelian
  differentials: the principal boundary, counting problems, and the
  {S}iegel-{V}eech constants}, Publ. Math. Inst. Hautes \'Etudes Sci. (2003),
  no.~97, 61--179.

\bibitem[Fil]{Fi2}
Simion Filip, \emph{Semisimplicity and rigidity of the {K}ontsevich-{Z}orich
  cocycle}, preprint, arXiv:1307.7314 (2013).

\bibitem[Fil16]{Fi1}
\bysame, \emph{Splitting mixed {H}odge structures over affine invariant
  manifolds}, Ann. of Math. (2) \textbf{183} (2016), no.~2, 681--713.

\bibitem[Gad]{Gadre}
Vaibhav Gadre, \emph{Partial sums of excursions along random geodesics and
  volume asymptotics for thin parts of moduli spaces of quadratic
  differentials}, preprint, arXiv:1408.5812 (2014).

\bibitem[Gen]{Gen}
Quentin Gendron, \emph{The {D}eligne-{M}umford and the {I}ncidence {V}ariety
  {C}ompactifications of the {S}trata of {$\Omega\mathcal{M}_{g}$}}, preprint,
  arXiv:1503.03338 (2015).

\bibitem[HLM09]{HLM-S}
Pascal Hubert, Erwan Lanneau, and Martin M{\"o}ller, \emph{{${\rm GL}^+_2(\Bbb
  R)$}-orbit closures via topological splittings}, Surveys in differential
  geometry. {V}ol. {XIV}. {G}eometry of {R}iemann surfaces and their moduli
  spaces, Surv. Differ. Geom., vol.~14, Int. Press, Somerville, MA, 2009,
  pp.~145--169.

\bibitem[HLM12]{HLM-Q}
\bysame, \emph{Completely periodic directions and orbit closures of many
  pseudo-{A}nosov {T}eichmueller discs in {Q}(1,1,1,1)}, Math. Ann.
  \textbf{353} (2012), no.~1, 1--35.

\bibitem[HM79]{HubbMas}
John Hubbard and Howard Masur, \emph{Quadratic differentials and foliations},
  Acta Math. \textbf{142} (1979), no.~3-4, 221--274.

\bibitem[HW]{HW}
W.~Patrick Hooper and Barak Weiss, \emph{Rel leaves of the {A}rnoux-{Y}occoz
  surfaces}, preprint, arXiv:1508.05363 (2015).

\bibitem[IT92]{IT}
Y.~Imayoshi and M.~Taniguchi, \emph{An introduction to {T}eichm\"uller spaces},
  Springer-Verlag, 1992.

\bibitem[KM]{KM}
Abhinav Kumar and Ronen~E. Mukamel, \emph{Algebraic models and arithmetic
  geometry of {T}eichm\"uller curves in genus two}, preprint, arXiv:1406.7057
  (2014).

\bibitem[KZ03]{KZ}
Maxim Kontsevich and Anton Zorich, \emph{Connected components of the moduli
  spaces of {A}belian differentials with prescribed singularities}, Invent.
  Math. \textbf{153} (2003), no.~3, 631--678.

\bibitem[LNW]{LNW}
Erwan Lanneau, Duc-Mahn Nguyen, and Alex Wright, \emph{Finiteness of
  {T}eichm\"uller curves in non-arithmetic rank 1 orbit closures}, preprint,
  arXiv:1504.03742 (2015).

\bibitem[LW15]{LeWe}
Samuel Leli{\`e}vre and Barak Weiss, \emph{Translation surfaces with no convex
  presentation}, Geom. Funct. Anal. \textbf{25} (2015), no.~6, 1902--1936.

\bibitem[Mas86]{MasurClosed}
Howard Masur, \emph{Closed trajectories for quadratic differentials with an
  application to billiards}, Duke Math. J. \textbf{53} (1986), no.~2, 307--314.

\bibitem[McM03]{Mc}
Curtis~T. McMullen, \emph{Billiards and {T}eichm\"uller curves on {H}ilbert
  modular surfaces}, J. Amer. Math. Soc. \textbf{16} (2003), no.~4, 857--885
  (electronic).

\bibitem[McM05]{McM:spin}
\bysame, \emph{Teichm\"uller curves in genus two: discriminant and spin}, Math.
  Ann. \textbf{333} (2005), no.~1, 87--130.

\bibitem[McM06a]{Mc2}
\bysame, \emph{Prym varieties and {T}eichm\"uller curves}, Duke Math. J.
  \textbf{133} (2006), no.~3, 569--590.

\bibitem[McM06b]{Mc4}
\bysame, \emph{Teichm\"uller curves in genus two: torsion divisors and ratios
  of sines}, Invent. Math. \textbf{165} (2006), no.~3, 651--672.

\bibitem[McM07]{Mc5}
\bysame, \emph{Dynamics of {${\rm SL}_2(\Bbb R)$} over moduli space in genus
  two}, Ann. of Math. (2) \textbf{165} (2007), no.~2, 397--456.

\bibitem[McM13]{McM:nav}
\bysame, \emph{Navigating moduli space with complex twists}, J. Eur. Math. Soc.
  (JEMS) \textbf{15} (2013), no.~4, 1223--1243.

\bibitem[McM14]{McM:iso}
\bysame, \emph{Moduli spaces of isoperiodic forms on {R}iemann surfaces}, Duke
  Math. J. \textbf{163} (2014), no.~12, 2271--2323.

\bibitem[McM15]{McCas}
\bysame, \emph{Cascades in the dynamics of measured foliations}, Ann. Sci.
  \'Ec. Norm. Sup\'er. (4) \textbf{48} (2015), 1--39.

\bibitem[M{\"o}l06a]{M2}
Martin M{\"o}ller, \emph{Periodic points on {V}eech surfaces and the
  {M}ordell-{W}eil group over a {T}eichm\"uller curve}, Invent. Math.
  \textbf{165} (2006), no.~3, 633--649.

\bibitem[M{\"o}l06b]{M}
\bysame, \emph{Variations of {H}odge structures of a {T}eichm\"uller curve}, J.
  Amer. Math. Soc. \textbf{19} (2006), no.~2, 327--344 (electronic).

\bibitem[MT02]{MT}
Howard Masur and Serge Tabachnikov, \emph{Rational billiards and flat
  structures}, Handbook of dynamical systems, {V}ol.\ 1{A}, North-Holland,
  Amsterdam, 2002, pp.~1015--1089.

\bibitem[Mum95]{Mum}
David Mumford, \emph{Algebraic geometry. {I}}, Classics in Mathematics,
  Springer-Verlag, Berlin, 1995, Complex projective varieties, Reprint of the
  1976 edition.

\bibitem[MW]{FullRank}
Maryam Mirzakhani and Alex Wright, \emph{Full rank affine invariant
  submanifolds}, preprint, arXiv:1608.02147 (2016).

\bibitem[MW02]{MinW}
Yair Minsky and Barak Weiss, \emph{Nondivergence of horocyclic flows on moduli
  space}, J. Reine Angew. Math. \textbf{552} (2002), 131--177.

\bibitem[MZ]{MZag}
Martin M{\"o}ller and Don Zagier, \emph{Modular embeddings of {T}eichmueller
  curves}, preprint, arXiv:1503.05690 (2015).

\bibitem[Ngu11]{N}
Duc-Manh Nguyen, \emph{Parallelogram decompositions and generic surfaces in {$
  H^{\rm hyp}(4)$}}, Geom. Topol. \textbf{15} (2011), no.~3, 1707--1747.

\bibitem[NW14]{NW}
Duc-Manh Nguyen and Alex Wright, \emph{Non-{V}eech surfaces in
  {$\mathcal{H}^{\rm hyp}(4)$} are generic}, Geom. Funct. Anal. \textbf{24}
  (2014), no.~4, 1316--1335.

\bibitem[Raf07]{Rafi}
Kasra Rafi, \emph{Thick-thin decomposition for quadratic differentials}, Math.
  Res. Lett. \textbf{14} (2007), no.~2, 333--341.

\bibitem[SW04]{SW2}
John Smillie and Barak Weiss, \emph{Minimal sets for flows on moduli space},
  Israel J. Math. \textbf{142} (2004), 249--260.

\bibitem[SW07]{SWprobs}
\bysame, \emph{Finiteness results for flat surfaces: a survey and problem
  list}, Partially hyperbolic dynamics, laminations, and {T}eichm\"uller flow,
  Fields Inst. Commun., vol.~51, Amer. Math. Soc., Providence, RI, 2007,
  pp.~125--137.

\bibitem[Vee89]{V}
W.~A. Veech, \emph{Teichm\"uller curves in moduli space, {E}isenstein series
  and an application to triangular billiards}, Invent. Math. \textbf{97}
  (1989), no.~3, 553--583.

\bibitem[Wri14]{Wfield}
Alex Wright, \emph{The field of definition of affine invariant submanifolds of
  the moduli space of abelian differentials}, Geom. Topol. \textbf{18} (2014),
  no.~3, 1323--1341.

\bibitem[Wri15a]{Wcyl}
\bysame, \emph{Cylinder deformations in orbit closures of translation
  surfaces}, Geom. Topol. \textbf{19} (2015), no.~1, 413--438.

\bibitem[Wri15b]{Wsurvey}
\bysame, \emph{Translation surfaces and their orbit closures: {A}n introduction
  for a broad audience}, EMS Surv. Math. Sci. \textbf{2} (2015), no.~1,
  63--108.

\bibitem[Wri16]{Wbilliards}
\bysame, \emph{From rational billiards to dynamics on moduli spaces}, Bull.
  Amer. Math. Soc. (N.S.) \textbf{53} (2016), no.~1, 41--56.

\bibitem[Zor06]{Z}
Anton Zorich, \emph{Flat surfaces}, Frontiers in number theory, physics, and
  geometry. {I}, Springer, Berlin, 2006, pp.~437--583.

\end{thebibliography}
\bibliographystyle{amsalpha}
\end{document}